\documentclass{mcom-l}

\usepackage{amssymb}

\newtheorem{theorem}{Theorem}[section]
\newtheorem{lemma}[theorem]{Lemma}

\theoremstyle{definition}

\newtheorem{example}[theorem]{Example}

\theoremstyle{remark}
\newtheorem{remark}[theorem]{Remark}

\numberwithin{equation}{section}

\begin{document}

% \title[short text for running head]{full title}
\title[Analysis for stochastic time-space fractional diffusion equation]{Numerical analysis for stochastic time-space fractional diffusion equation driven by fractional Gaussion noise}

%    Only \author and \address are required; other information is
%    optional.  Remove any unused author tags.

%    author one information
% \author[short version for running head]{name for top of paper}
\author{Daxin Nie}
\address{School of Mathematics and Statistics, Gansu Key Laboratory of Applied Mathematics and Complex Systems, Lanzhou University, Lanzhou 730000, P.R. China}
\curraddr{}
\email{ndx1993@163.com}
\thanks{}

%    author two information
\author{Weihua Deng}
\address{School of Mathematics and Statistics, Gansu Key Laboratory of Applied Mathematics and Complex Systems, Lanzhou University, Lanzhou 730000, P.R. China}
\curraddr{}
\email{dengwh@lzu.edu.cn}
\thanks{This work was supported by the National Natural Science Foundation of China under
	Grant No. 12071195 and the AI and Big Data Funds under Grant No. 2019620005000775.}

%    \subjclass is required.

\date{}

\dedicatory{}

%    Abstract is required.
\begin{abstract}
	In this paper, we consider the strong convergence of the time-space fractional diffusion equation driven by fractional Gaussion noise with Hurst index $H\in(\frac{1}{2},1)$. A sharp regularity estimate of the mild solution and the numerical scheme constructed by finite element method for integral fractional Laplacian and backward Euler convolution quadrature for Riemann-Liouville time fractional derivative are proposed. With the help of inverse Laplace transform and fractional Ritz projection, we obtain the accurate error estimates in time and space. Finally, our theoretical results are accompanied by numerical experiments.
\end{abstract}

\maketitle

%    Text of article.
\section{Introduction}

In the framework of uncoupled CTRW, if both the second moment of the jump length and the mean waiting time diverge, it describes competition between subdiffusion and L\'evy flights. The equivalent microscopic model is based on the subordinated Langevin equation with stable noise.  The probability density function of the position of the particle motion is governed by the fractional Fokker-Planck equation with temporal and spatial fractional derivatives \cite{Deng_2020}. If the system is influenced by external fluctuating source term, e.g., fractional Gaussian noise, it has the form (\ref{equretosol0}). Here we focus on its numerical analysis.

Let $\mathbb{D}\subset \mathbb{R}^{d}$ $(d=1,2,3)$ be a bounded domain with smooth boundary and $\psi(x,t)$ the solution of
\begin{equation}\label{equretosol0}
	\left \{
	\begin{aligned}
		&\partial_{t} \psi(x,t)+\!_0\partial^{1-\alpha}_t\mathcal{A}^{s} \psi(x,t)
		=\dot{W}^{H}_{Q}(x,t) \qquad\quad(x,t)\in \mathbb{D}\times(0,T],\\
		&\psi(x,0)=\psi_0(x) \qquad\qquad\qquad\qquad\qquad \ \qquad\qquad x\in\mathbb{D},\\
		&\psi(x,t)=0 \qquad\qquad\qquad\qquad\qquad\qquad\qquad\ \ \, (x,t)\in \mathbb{D}^{c}\times [0,T],
	\end{aligned}
	\right .
\end{equation}
where $\mathcal{A}^{s}=(-\Delta)^{s}$ with a zero Dirichlet boundary condition and defined by \cite{Deng2018}
\begin{equation*}
	\mathcal{A}^{s}\psi=c_{d,s}{\rm P.V.}\int_{\mathbb{R}^d}\frac{\psi(x)-\psi(y)}{|x-y|^{d+2s}}dy,\qquad s\in (0,1)
\end{equation*}
with $c_{d,s}=\frac{2^{2s}s\Gamma(d/2+s)}{\pi^{d/2}\Gamma(1-s)}$; $\mathbb{D}^{c}$ means the complement of $\mathbb{D}$; $T$ denotes a fixed terminal time; $\partial_{t}$ is the first-order derivative of $t$; $\ _0\partial^{1-\alpha}_t$ is the Riemann-Liouville fractional derivative, defined by \cite{Podlubny_1999}
\begin{equation}
	_{0}\partial^{1-\alpha}_t\psi=\frac{1}{\Gamma(\alpha)}\frac{\partial}{\partial t}\int^t_{0}(t-\xi)^{\alpha-1}\psi(\xi)d\xi,\quad \alpha\in(0,1);
\end{equation}
$\dot{W}^{H}_{Q}$ denotes fractional Gaussian noise; $W^{H}_{Q}$ is fractional Gaussian process with  Hurst index $H\in (\frac{1}{2},1)$ and covariance operator $Q$ on a complete filtered probability space $(\Omega,\mathcal{F},\mathbb{P},\{\mathcal{F}_{t}\}_{t\geq 0})$ and it can be written as
\begin{equation*}
	W^{H}_{Q}(x,t)=\sum_{k=1}^{\infty}\sqrt{\varLambda_{k}}\phi_k(x)W^{H}_k(t),
\end{equation*}
where $\{(\varLambda_{k},\phi_{k})\}_{k=1}^{\infty}$ are eigenvalues and orthonormal eigenfunctions of the self-adjoint, nonnegative linear operator $Q$ on $\mathbb{H}=L^{2}(\mathbb{D})$; $W^{H}_k$, $k=1,2,\ldots,$ are independent one-dimensional fractional Brownian motion (fBm) process with Hurst index $H$. In this paper, we assume that $A^{-\rho}Q^{\frac{1}{2}}$ is a Hilbert-Schmidt operator on $\mathbb
{H}$, where $\rho$ is a real number, $A$ denotes the classical Laplace operator $-\Delta$ with a zero Dirichlet boundary condition, and its domain $\mathcal{D}(A)=H^{1}_{0}(\mathbb{D})\cap H^{2}(\mathbb{D})$.

Obviously, the problem \eqref{equretosol0} can be divided into the following two problems, i.e., a deterministic problem
\begin{equation}\label{equretosol1}
	\left \{
	\begin{aligned}
		&\partial_{t} v+\!_0\partial^{1-\alpha}_t\mathcal{A}^{s}v=0
		\qquad\qquad\qquad~~~~\, (x,t)\in \mathbb{D}\times(0,T],\\
		&v(\cdot,0)=\psi_0 \qquad\qquad\qquad\qquad\ x\in\mathbb{D},\\
		&v=0 \qquad\qquad\qquad\qquad\qquad\ \ \ (x,t)\in \mathbb{D}^{c}\times[0,T],
	\end{aligned}
	\right .
\end{equation}
and a stochastic problem
\begin{equation}\label{equretosol}
	\left \{
	\begin{aligned}
		&\partial_{t} u+\!_0\partial^{1-\alpha}_t\mathcal{A}^{s}u
		=\dot{W}^{H}_{Q} \qquad\qquad~~ (x,t)\in \mathbb{D}\times(0,T],\\
		&u(\cdot,0)=0 \qquad\qquad\qquad\qquad\ x\in\mathbb{D},\\
		&u=0 \qquad\quad\qquad\qquad\qquad\ \ \ \ (x,t)\in \mathbb{D}^{c}\times[0,T].
	\end{aligned}
	\right .
\end{equation}
Extensive numerical schemes for the deterministic fractional diffusion equation \eqref{equretosol1} have been proposed in \cite{Acosta20173,Bonito2018,Nie2020}. Also, there have been many works for numerically solving stochastic partial differential equations (PDEs)  involving Laplace and spectral fractional Laplacian; one can refer to \cite{Bolin2020,kovacs_2014,kovacs_20142,li_2017,Li_2019,Liu_2020,nie20202,wang2017}. But for stochastic PDEs  involving integral fractional Laplacian, the related researches are still few. In this paper, we provide a numerical scheme for stochastic PDE \eqref{equretosol} based on backward Euler convolution for Riemann-Liouville fractional derivative and finite element method for integral fractional Laplacian.%As we know, \cite{hao2020} provides a spectral Galerkin scheme for solving the one-dimensional fractional elliptic equation driven by fractional Gaussian noise, etc..

Different from the Laplace and spectral fractional Laplacians, the eigenfunctions of $\mathcal{A}^{s}$ are unknown, so how to well characterize the influence of the noise on the regularity of the solution for \eqref{equretosol} is a challenge. Here we provide the sharp regularity of the mild solution of \eqref{equretosol} by building the resolvent estimate of integral fractional Laplacian and using the equivalence of the fractional Sobolev spaces, i.e., for $s\in[0,1]$, the spaces $\hat{H}^{s}(\mathbb{D})$ and $\dot{H}^{s}(\mathbb{D})$ are equivalent. Then we transform the Wiener integral with respect to fBm into the one of Brownian motion and use It\^{o}'s isometry to obtain error estimate for time semi-discrete scheme. Finally, we introduce the fractional Ritz projection to build the estimate $\|(\mathcal{A}^{s}_{h})^{-\frac{1}{2}}P_{h}A^{\frac{s}{2}}\|\leq C$ (see Section 5) to get accurate error estimates.

The paper is organized as follows. In Section 2, some preliminaries about fBm and fractional Sobolev spaces are introduced. In Section 3, we provide the spatial regularity estimate and H\"{o}lder regularity  estimate about the mild solution of Eq. \eqref{equretosol}. In Section 4, the Riemann-Liouville fractional derivative is approximated by backward Euler convolution quadrature method   and we provide the error estimates for the semidiscrete scheme. In Section 5,  finite element method is used to discretize integral fractional Laplacian and error estimates for the fully discrete scheme are provided. In Section 6, extensive numerical examples verify the theoretically predicted convergence order. We conclude the paper with some discussions in the last section.

Throughout this paper, $C$ denotes a generic positive constant, whose value may differ at each occurrence. The notation `$\tilde{~}$' means taking Laplace transform and let $\epsilon > 0$ be arbitrarily small.

\section{Preliminaries}
%\subsection{Preliminaries}
We provide some facts on fBm and fractional Sobolev spaces, which can refer to \cite{Mishura_2008,Kloeden_1995,Prato_2014}.

Introduce $\mathbb{H}_{0}=Q^{\frac{1}{2}}(\mathbb{H})$, whose inner product is $( \mu,\nu)_{\mathbb{H}_{0}}=( Q^{-\frac{1}{2}}\mu,Q^{-\frac{1}{2}}\nu)$ for $\mu,\nu\in\mathbb{H}_{0}$. Denote all the bounded linear operators from $\mathbb{H}$ to $\mathbb{H}$ and the ones from $\mathbb{H}_{0}$ to $\mathbb{H}$ by $\mathcal{L}(\mathbb{H})$ and $\mathcal{L}(\mathbb{H}_{0},\mathbb{H})$, respectively. The subspaces of $\mathcal{L}(\mathbb{H})$ and $\mathcal{L}(\mathbb{H}_{0},\mathbb{H})$ consisting of Hilbert-Schmidt operators are defined by $\mathcal{L}_{2}$ and $\mathcal{L}^{0}_{2}$ with norms, respectively, given by %, whose norms are defined by
\begin{equation*}
	\begin{aligned}
		&\|\mathcal{S}\|^{2}_{\mathcal{L}_{2}}= \langle \mathcal{S},\mathcal{S} \rangle_{\mathcal{L}_{2}}=\sum_{j\in\mathbb{N}}( \mathcal{S}\eta_{j},\mathcal{S}\eta_{j})_{\mathbb{H}},\quad \mathcal{S}\in \mathcal{L}_{2},\\
		&\|\mathcal{T}\|^{2}_{\mathcal{L}^{0}_{2}}= \langle \mathcal{T},\mathcal{T} \rangle_{\mathcal{L}_{2}}=\sum_{j\in\mathbb{N}}( \mathcal{T}\bar{\eta}_{j},\mathcal{T}\bar{\eta}_{j})_{\mathbb{H}},\quad \mathcal{T}\in \mathcal{L}^{0}_{2},\\
	\end{aligned}
\end{equation*}
which are independent of the specific choice of orthonormal basis $\{\eta_{j}\}_{j\in \mathbb{N}}$ in $\mathbb{H}$ and $\{\bar{\eta}_{j}\}_{j\in \mathbb{N}}$ in $\mathbb{H}_{0}$.
We denote $W^{H}_{Q}(x,t)$ as  $W^{H}_{Q}(t)$ and $\mathbb{E}$ as expectation operator in the following. Define $\mathbb{H}=L^{2}(\mathbb{D})$ with inner product $(\cdot,\cdot)$ and abbreviate $\|\cdot\|_{\mathcal{L}(\mathbb{H})}$ as $\|\cdot\|$.

For any $ q\geq -1 $, denote the space $ \hat{H}^q(\mathbb{D})=\mathcal{D}(A^{q/2})$ \cite{Thomee_2006} with the norm given by
\begin{equation*}
	|\mu|^2_{\hat{H}^q(\mathbb{D})}=\|A^{q/2}\mu\|_{L^2(\mathbb{D})}=\left(\sum_{j=1}^{\infty}\varkappa_{j}^{q}(\mu,\varphi_{j})^{2}\right )^{\frac{1}{2}},
\end{equation*}
where  $ \{(\varkappa_{j},\varphi_j)\}_{j=1}^{\infty} $ are $A$'s eigenvalues ordered non-decreasingly and the corresponding eigenfunctions normalized in the $\mathbb{H} $ norm. The eigenvalues of Laplace operator satisfy the following estimates.

\begin{lemma}[\cite{laptev_1997,Peter1983}]\label{thmeigenvalue}
	Let $\mathbb{D}$ be a bounded domain in $\mathbb{R}^d \,(d=1,2,3)$, with volume $|\mathbb{D}|$. Denote $\varkappa_j$ as the $j$-th eigenvalue of the Dirichlet boundary problem for the Laplace operator $-\Delta$ in $\mathbb{D}$. There is, for all $j\geq 1$,
	\begin{equation*}
		\varkappa_{j} \geq \frac{C_{d} d}{d+2} j^{\frac{2}{d}}|\mathbb{D}|^{-\frac{2}{d}},
	\end{equation*}
	where $C_{d}=(2 \pi)^{2} B_{d}^{-\frac{2}{d}}$ and $B_d$ denotes the volume of the unit $d$-dimensional ball.
\end{lemma}

Then we recall some fractional Sobolev spaces \cite{Acosta20172,Acosta20173,Acosta20171,Adams2013,borthagaray2018,Di2012,Lions1974}. For a given open set $\mathbb{D}\subset \mathbb{R}^d$ $(d=1,2,3)$, the fractional Sobolev spaces $ H^s(\mathbb{D}) $ with $s>0$  are defined by
\begin{equation*}
	H^s(\mathbb{D})=\left\{w\in H^{\lfloor s\rfloor}(\mathbb{D}):|w|_{H^s(\mathbb{D})}^{2}=\int\!\int_{\mathbb{D}^2}\frac{|D^{\lfloor s\rfloor}w(x)-D^{\lfloor s\rfloor}w(y)|^2}{|x-y|^{d+2(s-\lfloor s\rfloor)}}dxdy<\infty\right\}
\end{equation*}
with the norm
\begin{equation*}
	\|w\|_{H^{s}(\mathbb{D})}=\left(\|w\|^{2}_{H^{\lfloor s\rfloor}(\mathbb{D})}+|w|^{2}_{H^s(\mathbb{D})}\right)^{\frac{1}{2}},
\end{equation*}
where $\lfloor s\rfloor$ means the biggest integer not larger than $s$ and $D^{\lfloor s\rfloor}$ is $\lfloor s\rfloor$-th order derivative.
Introduce the subspace of $H^s(\mathbb{R}^d)$, consisting of the functions supported in $\bar{\mathbb{D}}$ and $s\in(0,1)$ by \cite{Acosta20173,borthagaray2018}
\begin{equation*}
	\dot{H}^s(\mathbb{D})=\{w\in H^s(\mathbb{R}^d):{\bf supp} \ w\subset \bar{\mathbb{D}}\},
\end{equation*}
which can also be defined by interpolation when $\mathbb{D}$ is a Lipschitz domain \cite{Acosta20173,borthagaray2018}, i.e.,
\begin{equation*}
	\dot{H}^s(\mathbb{D})=[L^{2}(\mathbb{D}),H^{1}_{0}(\mathbb{D})]_{s}.
\end{equation*}
The dual space of $\dot{H}^{s}(\mathbb{D})$ is denoted as $H^{-s}(\mathbb{D})$.

\begin{remark}\label{Remk1}
	According to \cite{Acosta20173}, for $s\in(0,1)$,
	the norm of $\dot{H}^{s}(\mathbb{D})$ induced by inner product, i.e.,
	\begin{equation}\label{equdefHsOmega}
		\langle u,w\rangle_s:=c_{d,s}\int\int_{(\mathbb{R}^d\times\mathbb{R}^d)\backslash(\mathbb{D}^c\times\mathbb{D}^c)}\frac{(u(x)-u(y))(w(x)-w(y))}{|x-y|^{d+2s}}dydx,
	\end{equation}
	is a multiple of the $H^{s}(\mathbb{R}^d)$-seminorm. We can get that the $H^{s}(\mathbb{R}^d)$-seminorm is equivalent to the full $H^{s}(\mathbb{R}^d)$-norm on this space by using the fractional Poincar\'{e}-type inequality \cite{Di2012}.
\end{remark}

\begin{remark}\label{Remk2}
	It is well-known that $ \hat{H}^0(\mathbb{D})=L^2(\mathbb{D}) $, $\hat{H}^1(\mathbb{D})=H^1_0(\mathbb{D})$, and $\hat{H}^2(\mathbb{D})=H^2(\mathbb{D})\bigcap H^1_0(\mathbb{D})$. According to \cite{Thomee_2006}, when $\mathbb{D}$ is a Lipschitz domain, we have $\hat{H}^{-s}(\mathbb{D})=H^{-s}(\mathbb{D})$ with $s\in(0,1)$ and $\hat{H}^{s}(\mathbb{D})=\dot{H}^{s}(\mathbb{D})$ with $s\in[0,1]$. Moreover, $\hat{H}^{s}(\mathbb{D})=\dot{H}^{s}(\mathbb{D})=H^{s}(\mathbb{D})$ for $s\in(0,\frac{1}{2})$.
\end{remark}

Next, we recall the elliptic regularity of fractional Laplacian $\mathcal{A}^{s}$ in \cite{Acosta20173,Grubb2015}.

\begin{theorem}[\cite{Grubb2015}]\label{thmregdiri}
	Let $u\in \hat{H}^s(\mathbb{D})$ be the solution of the Dirichlet problem
	\begin{equation}\label{equdiripro}
		\left \{\begin{aligned}
			(-\Delta)^su&=g\quad {\rm in}\ \mathbb{D},\\
			u&=0\quad {\rm in}\ \mathbb{D}^c,
		\end{aligned}\right.
	\end{equation}
	where $\mathbb{D} \subset \mathbb{R}^d$ is a bounded domain with smooth boundary and $g\in H^{\sigma}(\mathbb{D})$ for some $\sigma\geq -s$ and $s\in (0,1)$.
	Then, there exists a constant $C$ such that
	\begin{equation*}
		|u|_{H^{s+\gamma}(\mathbb{R}^d)}\leq C\|g\|_{H^{\sigma}(\mathbb{D})},
	\end{equation*}
	where $\gamma=\min(s+\sigma,\frac{1}{2}-\epsilon)$ with $\epsilon>0$ arbitrarily small.
\end{theorem}

As for one-dimensional fBm process, we have the fact \cite{Liu_2020,Mishura_2008}
\begin{equation}\label{eqrepfbm}
	\int_{0}^{t}f(s)dW^{H}(s)=C_{H}\left (H-\frac{1}{2}\right )\int_{0}^{t}\int_{s}^{t}f(r)(r-s)^{H-\frac{3}{2}}\left(\frac{s}{r}\right)^{\frac{1}{2}-H}drdW(s), ~~ t\in[0,T],
\end{equation}
where $C_{H}=\left(\frac{2H\Gamma(\frac{3}{2}-H)}{\Gamma(H+\frac{1}{2})\Gamma(2-2H)}\right )^{\frac{1}{2}}$ and $W(s)=W^{\frac{1}{2}}(s)$ is Brownian process.

Lastly, we define two sectors $\Sigma_{\theta}$ and $\Sigma_{\theta,\kappa}$ with $\kappa>0$ and $\pi/2<\theta<\pi$ as
\begin{equation*}
	\begin{aligned}
		&\Sigma_{\theta}=\{z\in\mathbb{C}:z\neq 0,|\arg z|\leq \theta\},\ \Sigma_{\theta,\kappa}=\{z\in\mathbb{C}:|z|>\kappa,|\arg z|\leq \theta\},\\
	\end{aligned}
\end{equation*}
and the contour $\Gamma_{\theta,\kappa}$ is given by
\begin{equation*}
	\Gamma_{\theta,\kappa}=\{\varrho e^{\pm\mathbf{i}\theta}: \varrho\geq \kappa\}\bigcup\{\kappa e^{\mathbf{i}\phi}: |\phi|\leq \theta\},\quad \pi/2<\theta<\pi,~~\kappa>0,
\end{equation*}
where the circular arc is oriented counterclockwise and the two rays are oriented with an increasing imaginary part and $\mathbf{i}^2=-1$.

\section{A priori estimate of the solution for \eqref{equretosol}}

With the help of Laplace transform and inverse Laplace transform, we write the mild solution of Eq. \eqref{equretosol} as
\begin{equation}\label{eqrepsol}
	u=\int_{0}^{t}\mathcal{R}(t-s)dW^{H}_{Q}(s),
\end{equation}
where the operator $\mathcal{R}(t)$ is defined by
\begin{equation}\label{eqdefE}
	\mathcal{R}(t)=\frac{1}{2\pi\mathbf{i}}\int_{\Gamma_{\theta,\kappa}}e^{zt}z^{\alpha-1}(z^{\alpha}+\mathcal{A}^{s})^{-1}dz.
\end{equation}
\begin{lemma}\label{lemresolest}
	Let $\mathcal{A}^{s}$ be the fractional Laplacian with a zero Dirichlet boundary condition and $s\in(0,1)$. Assume $\alpha\in(0,1)$ and $z\in \Sigma_{\theta,\kappa}$ with $\pi/2< \theta< \pi$. Then it follows that
	\begin{equation*}
		\|(z^{\alpha}+\mathcal{A}^{s})^{-1}\|_{\hat{H}^{\sigma}(\mathbb{D})\rightarrow \hat{H}^{\sigma+2\mu s}(\mathbb{D})}\leq C|z|^{(\mu-1)\alpha},
	\end{equation*}
	where $\mu\in[0,\min(1,\frac{1-\sigma}{2s})]$ and $\sigma\in [-s,\frac{1}{2}-s)$. When $\sigma\in[\frac{1}{2}-s,0]$ with $s\in[\frac{1}{2},1)$, we have
	\begin{equation*}
		\|(z^{\alpha}+\mathcal{A}^{s})^{-1}\|_{\hat{H}^{\sigma}(\mathbb{D})\rightarrow L^{2}(\mathbb{D})}\leq C|z|^{(-\frac{\sigma}{2s}-1)\alpha}.
	\end{equation*}
\end{lemma}
\begin{proof}
	Let $(z^{\alpha}+\mathcal{A}^{s})u=f$ with $z\in \Sigma_{\theta,\kappa}$, $u=0$ in $\mathbb{D}^{c}$, and $f\in \hat{H}^{-s}(\mathbb{D})$. Then there holds
	\begin{equation*}
		\begin{aligned}
			|z^{\alpha}(u,u)+\langle u,u\rangle_{s}|=|(f,u)|\leq \|f\|_{\hat{H}^{-s}(\mathbb{D})}\|u\|_{\hat{H}^{s}(\mathbb{D})}.
		\end{aligned}
	\end{equation*}
	Using the fact $a|z|+b\leq C|az+b|$ with $a,b\geq 0$ and $z\in \Sigma_{\theta}$ \cite{FUJITA1991}, one has
	\begin{equation*}
		\begin{aligned}
			&\left |z^{\alpha}(u,u)+\langle u,u\rangle_{s}\right |\\
			=&\left |z^{\alpha}\|u\|^{2}_{L^{2}(\mathbb{D})}+\|u\|^{2}_{\hat{H}^{s}(\mathbb{D})}\right |\\
			\geq&C|z|^{\alpha}\|u\|^{2}_{L^{2}(\mathbb{D})}+C\|u\|^{2}_{\hat{H}^{s}(\mathbb{D})},
		\end{aligned}
	\end{equation*}
	which leads to
	\begin{equation*}
		\|u\|_{\hat{H}^{s}(\mathbb{D})}\leq C\|f\|_{\hat{H}^{-s}(\mathbb{D})},\quad |z|^{\alpha/2}\|u\|_{L^{2}(\mathbb{D})}\leq C\|f\|_{\hat{H}^{-s}(\mathbb{D})}.
	\end{equation*}
	Thus
	\begin{equation*}
		\|(z^{\alpha}+\mathcal{A}^{s})^{-1}\|_{\hat{H}^{-s}(\mathbb{D})\rightarrow \hat{H}^{s}(\mathbb{D})}\leq C,\quad \|(z^{\alpha}+\mathcal{A}^{s})^{-1}\|_{\hat{H}^{-s}(\mathbb{D})\rightarrow L^{2}(\mathbb{D})}\leq C|z|^{-\alpha/2}.
	\end{equation*}
	Due to the operator $\mathcal{A}^{s}:~H^{l}(\mathbb{R}^{d})\rightarrow H^{l-2s}(\mathbb{R}^{d})$ is a bounded and invertible operator \cite{Acosta20173} and $H^{-s}(\mathbb{R}^{d})\subset \hat{H}^{-s}(\mathbb{D})$, there holds
	\begin{equation*}
		\|\mathcal{A}^{s}(z^{\alpha}+\mathcal{A}^{s})^{-1}\|_{\hat{H}^{-s}(\mathbb{D})\rightarrow \hat{H}^{-s}(\mathbb{D})}\leq C.
	\end{equation*}
	Combining the fact $\mathcal{A}^{s}(z^{\alpha}+\mathcal{A}^{s})^{-1}=I-z^{\alpha}(z^{\alpha}+\mathcal{A}^{s})^{-1}$ with $I$ being an identity operator, we see
	\begin{equation}
		\|(z^{\alpha}+\mathcal{A}^{s})^{-1}\|_{\hat{H}^{-s}(\mathbb{D})\rightarrow \hat{H}^{-s}(\mathbb{D})}\leq C|z|^{-\alpha}.
	\end{equation}
	By interpolation property and $\|(z^{\alpha}+\mathcal{A}^{s})^{-1}\|_{\hat{H}^{\varsigma}(\mathbb{D})\rightarrow\hat{H}^{\varsigma}(\mathbb{D})}\leq C|z|^{-\alpha}$ for $\varsigma=\max(0,\frac{1}{2}-s)$ (see \cite{Acosta20173,Nie2020}), we get
	\begin{equation*}
		\|(z^{\alpha}+\mathcal{A}^{s})^{-1}\|_{\hat{H}^{\sigma}(\mathbb{D})\rightarrow \hat{H}^{\sigma}(\mathbb{D})}\leq C|z|^{-\alpha},\quad \|\mathcal{A}^{s}(z^{\alpha}+\mathcal{A}^{s})^{-1}\|_{\hat{H}^{\sigma}(\mathbb{D})\rightarrow \hat{H}^{\sigma}(\mathbb{D})}\leq C
	\end{equation*}
	for $\sigma\in[-s,\max(0,\frac{1}{2}-s))$ and $s\in(0,1)$; and
	\begin{equation*}
		\|(z^{\alpha}+\mathcal{A}^{s})^{-1}\|_{\hat{H}^{\sigma}(\mathbb{D})\rightarrow L^{2}(\mathbb{D})}\leq C|z|^{(-\frac{\sigma}{2s}-1)\alpha}\quad {\rm for}~s\in[\frac{1}{2},1)~{\rm and}~\sigma\in[\frac{1}{2}-s,0].
	\end{equation*}
	
	Combining Theorem \ref{thmregdiri}, the definition of $\dot{H}^s(\mathbb{D})$, and Remarks \ref{Remk1} and \ref{Remk2}, we have, for $s\in(0,1)$ and $\sigma\in[-s,\frac{1}{2}-s)$,
	\begin{equation*}
		\begin{aligned}
			&\|(z^{\alpha}+\mathcal{A}^{s})^{-1}\|_{\hat{H}^{\sigma}(\mathbb{D})\rightarrow H^{\sigma+2s}(\mathbb{R}^d)}\leq C
		\end{aligned}
	\end{equation*}
and
\begin{equation*}
	\|(z^{\alpha}+\mathcal{A}^{s})^{-1}\|_{\hat{H}^{\sigma}(\mathbb{D})\rightarrow \hat{H}^{\sigma+2s}(\mathbb{D})}\leq C\quad {\rm for}~ \sigma+2s\in[0,1].
\end{equation*}
	Then interpolation properties give the desired results.
	%\begin{equation*}
	%	\|(z^{\alpha}+\mathcal{A}^{s})^{-1}\|_{\hat{H}^{\sigma}(\mathbb{D})\rightarrow \hat{H}^{1}(\mathbb{D})}\leq C|z|^{(\frac{1-\sigma}{2s}-1)\alpha}.
	%\end{equation*}
	
\end{proof}

Further we can obtain the spatial regularity estimate of $u$.
\begin{theorem}\label{thmsobo}
	Let $u$ be the mild soltuion of Eq. $\eqref{equretosol}$ and $\|A^{-\rho}\|_{\mathcal{L}^{0}_{2}}<\infty$ with $\rho\in(\frac{s}{2}-\frac{1}{4},\min(\frac{s}{2},\frac{sH}{\alpha}-\epsilon)]$.
	%where $\alpha\in(0,1)$.
	Then there exists a positive constant $C$ such that
	\begin{equation*}
		\mathbb{E}\|A^{\sigma}u\|^{2}_{\mathbb{H}}\leq C,
	\end{equation*}
	where $2\sigma\leq \min(2s-2\rho,\frac{2sH}{\alpha}-2\rho-\epsilon,1)$ with $\epsilon>0$ arbitrarily small.
\end{theorem}

\begin{proof}
	Eq. \eqref{eqrepsol} and It\^{o}'s isometry give
	\begin{equation*}
		\begin{aligned}
			&\mathbb{E}\|A^{\sigma}u\|^{2}_{\mathbb{H}}\\
			= &\mathbb{E}\left \|\int_{0}^{t}A^{\sigma}\mathcal{R}(t-r)dW^{H}_{Q}(r)\right \|^{2}_{\mathbb{H}}\\
			\leq&C\mathbb{E}\left \|\int_{0}^{t}\int_{r}^{t}A^{\sigma}\mathcal{R}(t-r')(r'-r)^{H-\frac{3}{2}}\left(\frac{r}{r'}\right)^{\frac{1}{2}-H}dr'dW_{Q}(r)\right \|^{2}_{\mathbb{H}}\\
			\leq&C\int_{0}^{t}\left \|\int_{r}^{t}A^{\sigma}\mathcal{R}(t-r')(r'-r)^{H-\frac{3}{2}}\left(\frac{r}{r'}\right)^{\frac{1}{2}-H}dr'\right \|^{2}_{\mathcal{L}^{0}_{2}}dr.
		\end{aligned}
	\end{equation*}
	Using $H\in(\frac{1}{2},1)$, the condition $\|A^{-\rho}\|_{\mathcal{L}^{0}_{2}}< \infty$, $\eqref{eqdefE}$, and Lemma \ref{lemresolest}, we find that
	\begin{equation*}
		\begin{aligned}	
			\mathbb{E}\|A^{\sigma}u\|^{2}_{\mathbb{H}}\leq&Ct^{2H-1}\int_{0}^{t}\left \|\int_{r}^{t}A^{\sigma}\mathcal{R}(t-r')(r'-r)^{H-\frac{3}{2}}A^{\rho}dr'\right \|^{2}r^{1-2H}dr\\
			\leq&Ct^{2H-1}\int_{0}^{t}\left \|\int_{r}^{t}\mathcal{R}(t-r')(r'-r)^{H-\frac{3}{2}}dr'\right \|^{2}_{\hat{H}^{-2\rho}(\mathbb{D})\rightarrow \hat{H}^{2\sigma}(\mathbb{D})}r^{1-2H}dr\\
			%\leq&Ct^{2H-1}\int_{0}^{t}\left \|\int_{\Gamma_{\theta,\kappa}}e^{z(t-r)}z^{\alpha-1}(z^{\alpha}+\mathcal{A}^{s})^{-1}z^{\frac{1}{2}-H}dz\right \|^{2}_{\hat{H}^{-2\rho}(\mathbb{D})\rightarrow \hat{H}^{2\sigma}(\mathbb{D})}r^{1-2H}dr\\
			\leq&Ct^{2H-1}\int_{0}^{t}\Big (\int_{\Gamma_{\theta,\kappa}}|e^{z(t-r)}||z|^{\alpha-1}\|(z^{\alpha}+\mathcal{A}^{s})^{-1}\|_{\hat{H}^{-2\rho}(\mathbb{D})\rightarrow \hat{H}^{2\sigma}(\mathbb{D})}\\
			&\qquad\qquad\qquad\qquad\qquad\qquad\cdot|z|^{\frac{1}{2}-H}|dz|\Big )^{2}r^{1-2H}dr\\
			\leq&Ct^{2H-1}\int_{0}^{t}\left (\int_{\Gamma_{\theta,\kappa}}|e^{z(t-r)}||z|^{\frac{\rho+\sigma}{s}\alpha-1+\frac{1}{2}-H}|dz|\right )^{2}r^{1-2H}dr\\
			\leq&Ct^{2H-1}\int_{0}^{t}(t-r)^{2(H-\frac{1}{2}-\frac{\rho+\sigma}{s}\alpha)}r^{1-2H}dr,\\
		\end{aligned}
	\end{equation*}
	where it is required that $-2\rho\in[-s,\frac{1}{2}-s)$ and $2\sigma\leq\min(2s-2\rho,1)$. Moreover, to preserve the boundness of $\mathbb{E}\|A^{\sigma}u\|^{2}_{\mathbb{H}}$, we need $H-\frac{1}{2}-\frac{\rho+\sigma}{s}\alpha>-\frac{1}{2}$, i.e., $0<\sigma<\frac{sH}{\alpha}-\rho$. Hence, the proof is completed.
\end{proof}

Lastly, we provide the H\"{o}lder regularity of the mild solution $u$.
\begin{theorem}\label{thmholder}
	Let $u$ be the mild solution of Eq. $\eqref{equretosol}$ and $\|A^{-\rho}\|_{\mathcal{L}^{0}_{2}}<\infty$ with $\rho\in[0,\min(\frac{s}{2},\frac{sH}{\alpha}-\epsilon)]$. Then $u$ satisfies
	\begin{equation*}
		\mathbb{E}\left \|u(t)-u(t-\tau)\right \|^{2}_{\mathbb{H}}\leq C\tau^{2\gamma},
	\end{equation*}
	where $\gamma\in(0, H-\frac{\rho\alpha}{s})$.
\end{theorem}
\begin{proof}
	Using
	\eqref{eqrepfbm}, we arrive at
	\begin{equation*}
		\begin{aligned}
			&\mathbb{E}\left\|\frac{u(t)-u(t-\tau)}{\tau^{\gamma}}\right \|_{\mathbb{H}}^{2}\\
			=&\mathbb{E}\left \|\frac{1}{\tau^{\gamma}}\left(\int_{0}^{t}\mathcal{R}(t-r)dW^{H}_{Q}(r)-\int_{0}^{t-\tau}\mathcal{R}(t-\tau-r)dW^{H}_{Q}(r)\right )\right \|_{\mathbb{H}}^{2}\\
			\leq &C\mathbb{E}\Big \|\frac{1}{\tau^{\gamma}}
			\Big(\int_{0}^{t}\int_{r}^{t}\mathcal{R}(t-r')(r'-r)^{H-\frac{3}{2}}\left (\frac{r}{r'}\right )^{\frac{1}{2}-H}dr'dW_{Q}(r)\\
			&\quad-\int_{0}^{t-\tau}\int_{r}^{t-\tau}\mathcal{R}(t-\tau-r')(r'-r)^{H-\frac{3}{2}}\left (\frac{r}{r'}\right )^{\frac{1}{2}-H}dr'dW_{Q}(r)\Big)\Big \|_{\mathbb{H}}^{2}\\
			\leq&C\mathbb{E}\Bigg \|\frac{1}{\tau^{\gamma}}\int_{0}^{t-\tau}\int_{r}^{t-\tau}(\mathcal{R}(t-r')-\mathcal{R}(t-\tau-r'))\\
			&\qquad\qquad\qquad\qquad\qquad\qquad\cdot(r'-r)^{H-\frac{3}{2}}\left (\frac{r}{r'}\right )^{\frac{1}{2}-H}dr'dW_{Q}(r)\Bigg \|_{\mathbb{H}}^{2}\\
			&+C\mathbb{E}\Bigg \|\frac{1}{\tau^{\gamma}}\int_{t-\tau}^{t}\int_{r}^{t}\mathcal{R}(t-r')(r'-r)^{H-\frac{3}{2}}\left (\frac{r}{r'}\right )^{\frac{1}{2}-H}dr'dW_{Q}(r)\Bigg \|_{\mathbb{H}}^{2}\\
			&+C\mathbb{E}\Bigg \|\frac{1}{\tau^{\gamma}}\int_{0}^{t-\tau}\int_{t-\tau}^{t}\mathcal{R}(t-r')(r'-r)^{H-\frac{3}{2}}\left (\frac{r}{r'}\right )^{\frac{1}{2}-H}dr'dW_{Q}(r)\Bigg \|_{\mathbb{H}}^{2}\\
			\leq&\uppercase\expandafter{\romannumeral1}+\uppercase\expandafter{\romannumeral2}+\uppercase\expandafter{\romannumeral3}.
		\end{aligned}
	\end{equation*}
	
Using It\^{o}'s isometry \cite{Peter1983,Mishura_2008}, Lemma \ref{lemresolest}, and $\left |\frac{e^{z\tau}-1}{\tau^{\gamma}}\right |\leq |z|^{\gamma}$ with $\gamma>0$ \cite{gunzburger_2018} gives
	\begin{equation*}
		\begin{aligned}
			&\uppercase\expandafter{\romannumeral1}\\
			\leq& C(t-\tau)^{2H-1}\int_{0}^{t-\tau}\left\|\int_{r}^{t-\tau}\frac{\mathcal{R}(t-r')-\mathcal{R}(t-\tau-r')}{\tau^{\gamma}}(r'-r)^{H-\frac{3}{2}}dr'\right \|_{\mathcal{L}^{0}_{2}}^{2}r^{1-2H}dr\\
		%	\leq& C(t-\tau)^{2H-1}\int_{0}^{t-\tau}\left\|\int_{r}^{t-\tau}\frac{\mathcal{R}(t-r')-\mathcal{R}(t-\tau-r')}{\tau^{\gamma}}A^{\rho}(r'-r)^{H-\frac{3}{2}}dr'\right \|^{2}r^{1-2H}dr\\
			\leq& C(t-\tau)^{2H-1}\int_{0}^{t-\tau}\left\|\int_{\Gamma_{\theta, \kappa}}e^{z(t-\tau-r)}\frac{e^{z\tau}-1}{\tau^{\gamma}}z^{\alpha-\frac{1}{2}-H}(z^{\alpha}+\mathcal{A}^{s})^{-1}A^{\rho}dz\right \|^{2}r^{1-2H}dr\\
			\leq& C(t-\tau)^{2H-1}\int_{0}^{t-\tau}(t-\tau-r)^{2(H-\frac{1}{2}-\gamma-\frac{\rho\alpha}{s})}r^{1-2H}dr,
		\end{aligned}
	\end{equation*}
	where the last inequality holds when $-2\rho\in [-s,0]$ and $\gamma>0$. To preserve the boundness of $\uppercase\expandafter{\romannumeral1}$, it also needs that $\gamma\in(0,H-\frac{\rho\alpha}{s})$.
	
	Using It\^{o}'s isometry and Lemma \ref{lemresolest} again, we get
	\begin{equation*}
		\begin{aligned}
			\uppercase\expandafter{\romannumeral2}\leq& C\frac{1}{\tau^{2\gamma}}t^{2H-1}\int_{t-\tau}^{t}\left \|\int_{r}^{t}\mathcal{R}(t-r')(r'-r)^{H-\frac{3}{2}}dr'\right \|_{\mathcal{L}^{0}_{2}}^{2}r^{1-2H}dr\\
			\leq& C\frac{1}{\tau^{2\gamma}}t^{2H-1}\int_{t-\tau}^{t}\left \|\int_{\Gamma_{\theta, \kappa}}e^{z(t-r)}z^{\alpha-1}(z^{\alpha}+\mathcal{A}^{s})^{-1}A^{\rho}z^{\frac{1}{2}-H}dz\right \|^{2}r^{1-2H}dr\\
			\leq& C\frac{1}{\tau^{2\gamma}}t^{2H-1}\int_{t-\tau}^{t}(t-r)^{2(H-\frac{1}{2}-\frac{\rho\alpha}{s})}r^{1-2H}dr,\\
		\end{aligned}
	\end{equation*}
	where the last inequality holds when $-2\rho\in [-s,0]$, and we require $\gamma\in(0,H-\frac{\rho\alpha}{s})$ to ensure the boundness of $\uppercase\expandafter{\romannumeral2}$.
	
	Similarly, for $\uppercase\expandafter{\romannumeral3}$, we obtain
	\begin{equation*}
		\begin{aligned}
			\uppercase\expandafter{\romannumeral3}\leq&C\frac{1}{\tau^{2\gamma}}\int_{0}^{t-\tau}\left \|\int_{t-\tau}^{t}\mathcal{R}(t-r')(r'-r)^{H-\frac{3}{2}}A^{\rho}r'^{H-\frac{1}{2}}dr'\right \|^{2}r^{1-2H}dr\\
			\leq&C\frac{1}{\tau^{2\gamma}}t^{2H-1}\int_{0}^{t-\tau}(t-\tau-r)^{-1+2\epsilon}r^{1-2H}\\
			&\qquad\qquad\qquad\qquad\qquad\cdot\left \|\int_{t-\tau}^{t}\mathcal{R}(t-r')A^{\rho}(r'-(t-\tau))^{H-1-\epsilon}dr'\right \|^{2}dr\\
			\leq &C\frac{1}{\tau^{2\gamma}}\tau^{2(H-\epsilon-\frac{\rho\alpha}{s})},
		\end{aligned}
	\end{equation*}
	where the above inequalities hold when $-2\rho\in [-s,0]$ and $\gamma>0$. We require $\gamma\in(0,H-\frac{\rho\alpha}{s})$ to preserve the boundness of $\uppercase\expandafter{\romannumeral3}$. Thus, combining $\uppercase\expandafter{\romannumeral1}$, $\uppercase\expandafter{\romannumeral2}$, and $\uppercase\expandafter{\romannumeral3}$ leads to the desired estimate.
\end{proof}
\iffalse
Combining Theorems \ref{thmsobo}, \ref{thmholder} and the regularity results for Eq. \eqref{equretosol1} in \cite{jin2016-1}, we obtain
\begin{theorem}
	Let $G$ be the mild solution of Eq. \eqref{equretosol0} and $\|A^{-\rho}\|_{\mathcal{L}^{0}_{2}}<\infty$. We have
	\begin{enumerate}
		\item if $\rho<\frac{H}{\alpha}$ and $G_{0}\in \hat{H}^{q}(\mathbb{D})$ with $q\leq \sigma$, then
		\begin{equation*}
			(\mathbb{E}\|A^{\sigma}G\|_{\mathbb{H}}^{2})^{\frac{1}{2}}\leq C+Ct^{-(\sigma-q)\alpha}\|G_{0}\|_{\hat{H}^{q}(\mathbb{D})},
		\end{equation*}
		where $\sigma=\min(1,1-\rho,\frac{H}{\alpha}-\rho-\epsilon)$.
		\item if $\rho\in[0,\frac{H}{\alpha})\cap[0,1]$ and $G_{0}\in \mathbb{H}$, then
		\begin{equation*}
			\left (\mathbb{E}\left \|\frac{G(t)-G(t-\tau)}{\tau^{\gamma}}\right \|_{\mathbb{H}}^{2}\right )^{\frac{1}{2}}\leq C+Ct^{-\gamma}\|G_{0}\|_{\mathbb{H}},
		\end{equation*}
		where $\gamma<H-\rho\alpha$.
	\end{enumerate}
\end{theorem}
\fi
\section{Time discretization and error analysis}
In this section, we turn to the discretization in time. Backward Euler convolution quadrature introduced in \cite{lubich_1988-1,lubich_1988} is used to discretize the Riemann-Liouville fractional derivative; and with the help of It\^{o}'s isometry, we obtain the corresponding error estimates of the temporal semi-discrete scheme.

Denote time step size $\tau=\frac{T}{N}$ ($N\in\mathbb{N}$) and $t_i=i\tau$, $i=0,1,\ldots,N$, where $T$ is a fixed terminal time. Using backward Euler convolution quadrature method \cite{lubich_1988-1,lubich_1988,lubich_1996}, we have the semi-discrete scheme of \eqref{equretosol} as
\begin{equation}\label{eqfullscheme}
	\frac{u^{n}-u^{n-1}}{\tau}+\sum_{i=0}^{n-1}d^{(1-\alpha)}_{i}\mathcal{A}^{s}u^{n-i}=\bar{\partial}_{\tau} W^{H}_{Q}(t_{n}),
\end{equation}
where $\{d_{i}^{(\alpha)}\}_{i=0}^{\infty}$ can be obtained by
\begin{equation}\label{eqdefd}
	(\delta_{\tau}(\zeta))^{\alpha}=\left(\frac{1-\zeta}{\tau}\right )^{\alpha}=\sum_{i=1}^{\infty}d^{(\alpha)}_{i}\zeta^{i},\quad \alpha\in(0,1)
\end{equation}
and
\begin{equation}
	\bar{\partial}_{\tau} W^{H}_{Q}(t)=\left\{
	\begin{aligned}
		&0\qquad t=t_{0},\\
		&\frac{W^{H}_{Q}(t_{j})-W^{H}_{Q}(t_{j-1})}{\tau}\qquad t\in(t_{j-1},t_{j}],\\
		&0\qquad t>t_{N}.
	\end{aligned} \right.
\end{equation}

Introduce the notation `$\tilde{~}$' as Laplace transform. The fact \cite{gunzburger_2018}
\begin{equation*}
	\sum_{n=1}^{\infty}\bar{\partial}_{\tau} W^{H}_{Q}(t_{n})e^{-zt_{n}}=\frac{z}{e^{z\tau}-1} \widetilde{\bar{\partial}_{\tau}W^{H}_{Q}},
\end{equation*}
and simple calculations (which can refer to \cite{gunzburger_2018,nie20202}) lead to
that the solution of Eq. \eqref{eqfullscheme} can be represented by
\begin{equation*}
	u^{n}=\frac{1}{2\pi\mathbf{i}}\int_{\Gamma^{\tau}_{\theta,\kappa}}e^{zt_{n}}(\delta_{\tau}(e^{-z\tau}))^{\alpha-1}((\delta_{\tau}(e^{-z\tau}))^{\alpha}+\mathcal{A}^{s})^{-1}\frac{z\tau}{e^{z\tau}-1} \widetilde{\bar{\partial}_{\tau}W^{H}_{Q}}dz,
\end{equation*}
where $\Gamma^\tau_{\theta,\kappa}=\{z\in \mathbb{C}:\kappa\leq |z|\leq\frac{\pi}{\tau\sin(\theta)},|\arg z|=\theta\}\cup\{z\in \mathbb{C}:|z|=\kappa,|\arg z|\leq\theta\}$.
Let
\begin{equation*}
	\bar{\mathcal{R}}(t)=\frac{1}{2\pi\mathbf{i}}\int_{\Gamma^{\tau}_{\theta,\kappa}}e^{zt}(\delta_{\tau}(e^{-z\tau}))^{\alpha-1}((\delta_{\tau}(e^{-z\tau}))^{\alpha}+\mathcal{A}^{s})^{-1}\frac{z\tau}{e^{z\tau}-1} dz;
\end{equation*}
and the convolution property of Laplace transform gives
\begin{equation}\label{eqrepsolfull}
	u^{n}=\int_{0}^{t_{n}}\bar{\mathcal{R}}(t_{n}-s)\bar{\partial}_{\tau}W^{H}_{Q}(s)ds.
\end{equation}

Then we recall a lemma in \cite{gunzburger_2018}, which is used to estimate $\mathbb{E}\|u(t_{n})-u^{n}\|_{\mathbb{H}}^{2}$.
\begin{lemma}[\cite{gunzburger_2018}]\label{Lemseriesest}
	Let $\delta_{\tau}$ be defined in \eqref{eqdefd}, $\Gamma^\tau_\xi=\{z=-\ln{\xi}/\tau+\mathbf{i}y:y\in\mathbb{R}~and~|y|\leq \pi/\tau\}$ with a fixed $\xi \in (0,1)$, and $\theta \in\left(\frac{\pi}{2}, \operatorname{arccot}\left(-\frac{2}{\pi}\right)\right)$,  where $arccot$ denotes the inverse function of $\cot$. If $z$ lies in the region enclosed by $\Gamma^\tau_\xi$, $\Gamma^\tau_{\theta,\kappa}$, and the two lines $\mathbb{R}\pm \mathbf{i}\pi/\tau$ with $0<\kappa \leq \min (1 / T,-\ln (\xi) / \tau)$, then we have
	\begin{enumerate}
		\item $\delta_\tau(e^{-z\tau})$ and $(\delta_\tau(e^{-z\tau})+A)^{-1}$ are both analytic;
		\item there exist positive constants $C_0$, $C_1$, and $C$  such that
		\begin{equation*}
			\begin{aligned}
				&\delta_{\tau}\left(e^{-z \tau}\right) \in \Sigma_{\theta}&\forall z \in \Gamma_{\theta, \kappa}^{\tau},\\
				&C_{0}|z| \leq\left|\delta_{\tau}\left(e^{-z\tau }\right)\right| \leq C_{1}|z|&\forall z \in \Gamma_{\theta, \kappa}^{\tau},\\
				&\left|\delta_{\tau}\left(e^{-z\tau }\right)-z\right| \leq C \tau|z|^{2}&\forall z \in \Gamma_{\theta, \kappa}^{\tau},\\
				&\left|\delta_{\tau}\left(e^{-z\tau }\right)^{\alpha}-z^{\alpha}\right| \leq C \tau|z|^{\alpha+1}&\forall z \in \Gamma_{\theta, \kappa}^{\tau},
			\end{aligned}
		\end{equation*}
		where $\kappa\in (0,\min (1 / T,-\ln (\xi) / \tau))$. Here, $C_{0}$, $C_{1}$, and $C$ are independent of $\tau$.
	\end{enumerate}

\end{lemma}
In the following proof, we need to take $\kappa\leq\frac{\pi}{t_{n}|\sin(\theta)|}$.
Then we provide the error estimates for the time semi-discrete scheme \eqref{eqfullscheme}.

\begin{theorem}\label{thmtimeerror}
	Let $u(t)$ and $u^{n}$ be the solutions of Eqs. \eqref{equretosol} and \eqref{eqfullscheme}, respectively. Assume $\|A^{-\rho}\|_{\mathcal{L}^{0}_{2}}<\infty$ with $\rho\in[0,\min(\frac{s}{2},\frac{sH}{\alpha})]$. %, $s\in(0,1)$, $\alpha\in(0,1)$ and $H\in(\frac{1}{2},1)$.
	Then
	\begin{equation*}
		\mathbb{E}\|u(t_{n})-u^{n}\|_{\mathbb{H}}^{2}\leq C\tau^{2H-\frac{2\rho\alpha}{s}}.
	\end{equation*}
\end{theorem}
\begin{proof}
	Subtracting \eqref{eqrepsolfull} from  \eqref{eqrepsol} and taking the expectation for $\|u(t_{n})-u^{n}\|^{2}_{\mathbb{H}}$ yield
	\begin{equation}\label{equtimeerrorrep}
		\begin{aligned}
			&\mathbb{E}\|u(t_{n})-u^{n}\|_{\mathbb{H}}^{2}\\
			=&\mathbb{E}\left \|\int_{0}^{t_{n}}\mathcal{R}(t_{n}-r)dW^{H}_{Q}(r)-\int_{0}^{t_{n}}\bar{\mathcal{R}}(t_{n}-r)(\bar{\partial}_{\tau}W^{H}_{Q}(r))dr\right \|_{\mathbb{H}}^{2}\\
			\leq&\mathbb{E}\left \|\int_{0}^{t_{n}}(\mathcal{R}(t_{n}-r)-\bar{\mathcal{R}}(t_{n}-r))dW^{H}_{Q}(r)\right \|_{\mathbb{H}}^{2}\\
			&+\mathbb{E}\left \|\int_{0}^{t_{n}}\bar{\mathcal{R}}(t_{n}-r)(dW^{H}_{Q}(r)-\bar{\partial}_{\tau}W^{H}_{Q}(r)dr)\right \|_{\mathbb{H}}^{2}\\
			=&\vartheta_{1}+\vartheta_{2}.
		\end{aligned}
	\end{equation}
	Here, using \eqref{eqrepfbm} and It\^{o}'s isometry, we separate $\vartheta_{1}$ into two parts
	\begin{equation*}
		\begin{aligned}
			\vartheta_{1}\leq&C\mathbb{E}\left \|\int_{0}^{t_{n}}\int_{r}^{t_{n}}(\mathcal{R}(t_{n}-r')-\bar{\mathcal{R}}(t_{n}-r'))(r'-r)^{H-\frac{3}{2}}\left(\frac{r}{r'}\right )^{\frac{1}{2}-H}dr'dW_{Q}(r)\right \|_{\mathbb{H}}^{2}\\
			\leq&Ct_{n}^{2H-1}\int_{0}^{t_{n}}\left \|\int_{r}^{t_{n}}(\mathcal{R}(t_{n}-r')-\bar{\mathcal{R}}(t_{n}-r'))(r'-r)^{H-\frac{3}{2}}dr'\right \|_{\mathcal{L}^{0}_{2}}^{2}r^{1-2H}dr\\
			\leq &Ct_{n}^{2H-1}\int_{0}^{t_{n}}\left \|\int_{\Gamma_{\theta, \kappa}\backslash\Gamma^{\tau}_{\theta,\kappa}}e^{z(t_{n}-r)}z^{\alpha-1}(z^{\alpha}+\mathcal{A}^{s})^{-1}z^{\frac{1}{2}-H}A^{\rho}dz\right \|^{2}r^{1-2H}dr\\
			&+Ct_{n}^{2H-1}\int_{0}^{t_{n}}\Big \|\int_{\Gamma^{\tau}_{\theta,\kappa}}e^{z(t_{n}-r)}(z^{\alpha-1}(z^{\alpha}+\mathcal{A}^{s})^{-1}-\\
			&\qquad(\delta_{\tau}(e^{-z\tau}))^{\alpha-1}((\delta_{\tau}(e^{-z\tau}))^{\alpha}+\mathcal{A}^{s})^{-1}\frac{z\tau}{e^{z\tau}-1})z^{\frac{1}{2}-H}A^{\rho}dz\Big \|^{2}r^{1-2H}dr\\
			\leq& \vartheta_{1,1}+\vartheta_{1,2}.
		\end{aligned}
	\end{equation*}
	By $\rho\in[0,\min(\frac{s}{2},\frac{sH}{\alpha})]$, Lemma \ref{lemresolest}, It\^{o}'s isometry and Cauchy-Schwarz inequality, $\vartheta_{1,1}$ satisfies
	\begin{equation*}
		\begin{aligned}
			\vartheta_{1,1}\leq&Ct_{n}^{2H-1}\int_{0}^{t_{n}}\Bigg (\int_{\Gamma_{\theta, \kappa}\backslash\Gamma^{\tau}_{\theta, \kappa}}|e^{z(t_{n}-r)}||z|^{\alpha-\frac{1}{2}-H}\\
			&\qquad\qquad\qquad\qquad\cdot\|(z^{\alpha}+\mathcal{A}^{s})^{-1}\|_{\hat{H}^{-2\rho}(\mathbb{D})\rightarrow L^{2}(\mathbb{D})}|dz|\Bigg )^{2}r^{1-2H}dr\\
			\leq &Ct_{n}^{2H-1}\int_{0}^{t_{n}}\left(\int_{\Gamma_{\theta, \kappa}\backslash\Gamma^{\tau}_{\theta, \kappa}}|z|^{\frac{2\rho\alpha}{s}-1-2H}|dz|\int_{\Gamma_{\theta, \kappa}\backslash\Gamma^{\tau}_{\theta, \kappa}}|e^{2z(t_{n}-r)}||dz|r^{1-2H}\right)dr\\
			\leq &Ct_{n}^{2H-1}\tau^{2H-\frac{2\rho\alpha}{s}}\int_{0}^{t_{n}}\int_{\Gamma_{\theta, \kappa}\backslash\Gamma^{\tau}_{\theta, \kappa}}|e^{2z(t_{n}-r)}||dz|r^{1-2H}dr.
		\end{aligned}
	\end{equation*}
	Simple calculations and mean value theorem give
	\begin{equation}\label{eqspederi}
		\begin{aligned}
			&\int_{0}^{t_{n}}\int_{\Gamma_{\theta, \kappa}\backslash\Gamma^{\tau}_{\theta, \kappa}}|e^{2z(t_{n}-r)}||dz|r^{1-2H}dr\\
			\leq &\int_{\Gamma_{\theta, \kappa}\backslash\Gamma^{\tau}_{\theta, \kappa}}e^{2\cos(\theta)|z|t_{n}}\int_{0}^{\frac{t_{n}}{4}}e^{-2\cos(\theta)|z|r}r^{1-2H}dr|dz|\\
			&+\int_{\Gamma_{\theta, \kappa}\backslash\Gamma^{\tau}_{\theta, \kappa}}e^{2\cos(\theta)|z|t_{n}}\int_{\frac{t_{n}}{4}}^{t_{n}}e^{-2\cos(\theta)|z|r}r^{1-2H}dr|dz|\\
			\leq &C\int_{\Gamma_{\theta, \kappa}\backslash\Gamma^{\tau}_{\theta, \kappa}}e^{\cos(\theta)|z|t_{n}}\int_{0}^{\frac{t_{n}}{4}}e^{\cos(\theta)|z|r}r^{1-2H}dr|dz|\\
			&+Ct_{n}^{1-2H}\int_{\Gamma_{\theta, \kappa}\backslash\Gamma^{\tau}_{\theta, \kappa}}e^{2\cos(\theta)|z|t_{n}}\int_{\frac{t_{n}}{4}}^{t_{n}}e^{-2\cos(\theta)|z|r}dr|dz|\\
			\leq &C\int_{\Gamma_{\theta, \kappa}\backslash\Gamma^{\tau}_{\theta, \kappa}}|z|^{2H-2}e^{\cos(\theta)|z|t_{n}}|dz|\\
			&+Ct_{n}^{2-2H}\int_{\Gamma_{\theta, \kappa}\backslash\Gamma^{\tau}_{\theta, \kappa}}e^{C\cos(\theta)|z|t_{n}}|dz|\\
			\leq& Ct_{n}^{1-2H},
		\end{aligned}
	\end{equation}
	which implies
	\begin{equation*}
		\vartheta_{1,1}\leq C\tau^{2H-\frac{2\rho\alpha}{s}}.
	\end{equation*}
	By Lemmas \ref{lemresolest} and \ref{Lemseriesest}, we have
	\begin{equation*}
		\begin{aligned}
			&\left \|\left (z^{\alpha-1}(z^{\alpha}+\mathcal{A}^{s})^{-1}-(\delta_{\tau}(e^{-z\tau}))^{\alpha-1}((\delta_{\tau}(e^{-z\tau}))^{\alpha}+\mathcal{A}^{s})^{-1}\frac{z\tau}{e^{z\tau}-1}\right )A^{\rho}\right \|\\
			\leq&\|(z^{\alpha-1}(z^{\alpha}+\mathcal{A}^{s})^{-1}-(\delta_{\tau}(e^{-z\tau}))^{\alpha-1}(z^{\alpha}+\mathcal{A}^{s})^{-1})A^{\rho}\|\\
			&+\|((\delta_{\tau}(e^{-z\tau}))^{\alpha-1}(z^{\alpha}+\mathcal{A}^{s})^{-1}-(\delta_{\tau}(e^{-z\tau}))^{\alpha-1}((\delta_{\tau}(e^{-z\tau}))^{\alpha}+\mathcal{A}^{s})^{-1})A^{\rho}\|\\
			&+\left \|(\delta_{\tau}(e^{-z\tau}))^{\alpha-1}((\delta_{\tau}(e^{-z\tau}))^{\alpha}+\mathcal{A}^{s})^{-1}\left(1-\frac{z\tau}{e^{z\tau}-1}\right )A^{\rho}\right \|\\
			\leq &C\tau |z|^{\frac{\rho\alpha}{s}}.
		\end{aligned}
	\end{equation*}
	The above estimate and Cauchy-Schwarz inequality lead to
	\begin{equation*}
		\begin{aligned}
			\vartheta_{1,2}\leq&C\tau^{2} t_{n}^{2H-1}\int_{0}^{t_{n}}\left (\int_{\Gamma^{\tau}_{\theta,\kappa}}|e^{z(t_{n}-r)}||z|^{\frac{\rho\alpha}{s}+\frac{1}{2}-H}|dz|\right )^{2}r^{1-2H}dr\\
			\leq&C\tau^{2} t_{n}^{2H-1}\int_{\Gamma^{\tau}_{\theta,\kappa}}|z|^{\frac{2\rho\alpha}{s}+1-2H}|dz|\int_{0}^{t_{n}}\int_{\Gamma^{\tau}_{\theta,\kappa}}|e^{2z(t_{n}-r)}||dz|r^{1-2H}dr.
		\end{aligned}
	\end{equation*}
	Similar to the derivation of \eqref{eqspederi}, there holds
	\begin{equation*}
		\begin{aligned}
			&\int_{0}^{t_{n}}\int_{\Gamma^{\tau}_{\theta, \kappa}}|e^{2z(t_{n}-r)}||dz|r^{1-2H}dr\\
			\leq &\int_{\Gamma^{\tau}_{\theta, \kappa}}e^{2\cos(\theta)|z|t_{n}}\int_{0}^{\frac{t_{n}}{4}}e^{-2\cos(\theta)|z|r}r^{1-2H}dr|dz|\\
			&+\int_{\Gamma^{\tau}_{\theta, \kappa}}e^{2\cos(\theta)|z|t_{n}}\int_{\frac{t_{n}}{4}}^{t_{n}}e^{-2\cos(\theta)|z|r}r^{1-2H}dr|dz|\\
			%		\leq &C\int_{\Gamma^{\tau}_{\theta, \kappa}}e^{\cos(\theta)|z|t_{n}}\int_{0}^{\frac{t_{n}}{4}}e^{\cos(\theta)|z|r}r^{1-2H}dr|dz|\\
			%		&+Ct_{n}^{1-2H}\int_{\Gamma^{\tau}_{\theta, \kappa}}e^{2\cos(\theta)|z|t_{n}}\int_{\frac{t_{n}}{4}}^{t_{n}}e^{-2\cos(\theta)|z|r}dr|dz|\\
			%		\leq &C\int_{\Gamma^{\tau}_{\theta, \kappa}}|z|^{2H-2}e^{\cos(\theta)|z|t_{n}}|dz|\\
			%		&+Ct_{n}^{2-2H}\int_{\Gamma^{\tau}_{\theta, \kappa}}e^{C\cos(\theta)|z|t_{n}}|dz|\\
			\leq& Ct_{n}^{1-2H}.
		\end{aligned}
	\end{equation*}
	Thus
	\begin{equation*}
		\vartheta_{1,2}\leq C\tau^{2H-\frac{2\rho\alpha}{s}}.
	\end{equation*}
	As for $\vartheta_2$, by the definition of $\bar{\partial}_{\tau}W^{H}_{Q}(r)$, we have
	\begin{equation*}
		\begin{aligned}
			\vartheta_{2}\leq&C\mathbb{E}\left \|\int_{0}^{t_{n}}\bar{\mathcal{R}}(t_{n}-r)(dW^{H}_{Q}(r)-\bar{\partial}_{\tau}W^{H}_{Q}(r)dr)\right \|_{\mathbb{H}}^{2}\\
			\leq&C\mathbb{E}\left \|\sum_{i=1}^{n}\int_{t_{i-1}}^{t_{i}}\frac{1}{\tau}\int_{t_{i-1}}^{t_{i}}(\bar{\mathcal{R}}(t_{n}-r)-\bar{\mathcal{R}}(t_{n}-\xi))d\xi dW^{H}_{Q}(r)\right \|_{\mathbb{H}}^{2}\\
			\leq &C \mathbb{E}\left \|\int_{0}^{t_{n}}\frac{1}{\tau}\sum_{i=1}^{n}\chi_{(t_{i-1},t_{i}]}(r)\int_{t_{i-1}}^{t_{i}}(\bar{\mathcal{R}}(t_{n}-r)-\bar{\mathcal{R}}(t_{n}-\xi))d\xi dW^{H}_{Q}(r)\right \|_{\mathbb{H}}^{2},
		\end{aligned}
	\end{equation*}
	where $\chi_{(a,b)}$ means the characteristic function on $(a,b)$.
	Introduce
	\begin{equation*}
		\begin{aligned} \mathcal{G}(t_{n}-r)=&\frac{1}{\tau}\sum_{i=1}^{n}\chi_{(t_{i-1},t_{i}]}(r)\int_{t_{i-1}}^{t_{i}}(\bar{\mathcal{R}}(t_{n}-r)-\bar{\mathcal{R}}(t_{n}-\xi))d\xi.\\
		\end{aligned}
	\end{equation*}
After simple calculations, there exists
	\begin{equation*}
		\begin{aligned}
			\mathcal{G}(t_{n}-r)=&\frac{1}{2\pi\mathbf{i}\tau}\sum_{i=1}^{n}\chi_{(t_{i-1},t_{i}]}(r)\int_{t_{i-1}}^{t_{i}}\int_{\Gamma^{\tau}_{\theta, \kappa}}(e^{z(t_{n}-r)}-e^{z(t_{n}-\xi)})\\
			&\qquad\cdot(\delta_{\tau}(e^{-z\tau}))^{\alpha-1}((\delta_{\tau}(e^{-z\tau}))^{\alpha}+\mathcal{A}^{s})^{-1}\frac{z\tau}{e^{z\tau}-1}dzd\xi\\
			=&\frac{1}{2\pi\mathbf{i}}\sum_{i=1}^{n}\chi_{(t_{i-1},t_{i}]}(r)\\
			&\cdot \left(\int_{\Gamma^{\tau}_{\theta, \kappa}}e^{z(t_{n}-r)}(\delta_{\tau}(e^{-z\tau}))^{\alpha-1}((\delta_{\tau}(e^{-z\tau}))^{\alpha}+\mathcal{A}^{s})^{-1}\frac{z\tau}{e^{z\tau}-1}dz\right.\\
			&-\left.\int_{\Gamma^{\tau}_{\theta, \kappa}}e^{z(t_{n}-r)}e^{z(r-t_{i})}(\delta_{\tau}(e^{-z\tau}))^{\alpha-1}((\delta_{\tau}(e^{-z\tau}))^{\alpha}+\mathcal{A}^{s})^{-1}dz\right).
		\end{aligned}
	\end{equation*}
	Simple calculations, \eqref{eqrepfbm}, and $\|A^{-\rho}\|_{\mathcal{L}^{0}_{2}}<\infty$ give
	\begin{equation*}
		\begin{aligned}
			\vartheta_{2}\leq &C \mathbb{E}\left \|\int_{0}^{t_{n}}\mathcal{G}(t_{n}-r)dW^{H}_{Q}(r)\right \|_{\mathbb{H}}^{2}\\
			\leq&C\mathbb{E}\left \|\int_{0}^{t_{n}}\int_{r}^{t_{n}}\mathcal{G}(t_{n}-r')(r'-r)^{H-\frac{3}{2}}\left(\frac{r}{r'}\right)^{\frac{1}{2}-H}dr'dW_{Q}(r)\right \|^{2}_{\mathbb{H}}\\
			\leq&C\int_{0}^{t_{n}}\left \|\int_{r}^{t_{n}}\mathcal{G}(t_{n}-r')(r'-r)^{H-\frac{3}{2}}\left(\frac{r}{r'}\right)^{\frac{1}{2}-H}dr'\right \|^{2}_{\mathcal{L}^{0}_{2}}dr\\
			%\leq&Ct_{n}^{2H-1}\int_{0}^{t_{n}}\left \|\int_{r}^{t_{n}}\mathcal{G}(t_{n}-r')(r'-r)^{H-\frac{3}{2}}A^{\rho}dr'\right \|^{2}r^{1-2H}dr\\
			\leq&Ct_{n}^{2H-1}\int_{0}^{t_{n}}\left \|\int_{r}^{t_{n}}\mathcal{G}(t_{n}-r')(r'-r)^{H-\frac{3}{2}}dr'\right \|^{2}_{\hat{H}^{-2\rho}(\mathbb{D})\rightarrow L^{2}(\mathbb{D})}r^{1-2H}dr.
		\end{aligned}
	\end{equation*}
	By Lemmas \ref{lemresolest} and \ref{Lemseriesest}, there holds
	\begin{equation*}
		\begin{aligned}
			&\left\|\mathcal{G}(t_{n}-r)\right \|_{\hat{H}^{-2\rho}(\mathbb{D})\rightarrow L^{2}(\mathbb{D})}\\
			\leq&C\sum_{i=1}^{n}\chi_{(t_{i-1},t_{i}]}(r)\int_{\Gamma^{\tau}_{\theta, \kappa}}|e^{z(t_{n}-r)}|\left |1-\frac{z\tau}{e^{z\tau}-1}\right |\\
			&\qquad \cdot\|(\delta_{\tau}(e^{-z\tau}))^{\alpha-1}((\delta_{\tau}(e^{-z\tau}))^{\alpha}+\mathcal{A}^{s})^{-1}\|_{\hat{H}^{-2\rho}(\mathbb{D})\rightarrow L^{2}(\mathbb{D})}|dz|\\
			&+C\sum_{i=1}^{n}\chi_{(t_{i-1},t_{i}]}(r)\int_{\Gamma^{\tau}_{\theta, \kappa}}|e^{z(t_{n}-r)}|\left |1-e^{z(r-t_{i})}\right |\\
			&\qquad \cdot\|(\delta_{\tau}(e^{-z\tau}))^{\alpha-1}((\delta_{\tau}(e^{-z\tau}))^{\alpha}+\mathcal{A}^{s})^{-1}\|_{\hat{H}^{-2\rho}(\mathbb{D})\rightarrow L^{2}(\mathbb{D})}|dz|\\
			\leq&C\tau\int_{\Gamma^{\tau}_{\theta,\kappa}}|e^{z(t_{n}-r)}||z|^{\frac{\rho\alpha}{s}}|dz|.
		\end{aligned}
	\end{equation*}
	
	Combining above estimate, Lemma \ref{lemresolest}, and similar to the derivation of \eqref{eqspederi}, we can get
	\begin{equation*}
		\begin{aligned}
			&\left \|\int_{r}^{t}\mathcal{G}(t-r')(r'-r)^{H-\frac{3}{2}}dr'\right \|^{2}_{\hat{H}^{-2\rho}(\mathbb{D})\rightarrow L^{2}(\mathbb{D})}\\
			\leq&\left(\int_{r}^{t}\|\mathcal{G}(t-r')\|_{\hat{H}^{-2\rho}(\mathbb{D})\rightarrow L^{2}(\mathbb{D})}(r'-r)^{H-\frac{3}{2}}dr'\right)^{2} \\
			\leq&C\tau^{2}\left(\int_{r}^{t}\int_{\Gamma^{\tau}_{\theta,\kappa}}|e^{z(t-r')}||z|^{\frac{\rho\alpha}{s}}|dz|(r'-r)^{H-\frac{3}{2}}dr'\right)^{2}\\
			\leq&C\tau^{2}\int_{r}^{t}\left(\int_{\Gamma^{\tau}_{\theta,\kappa}}|e^{z(t-r')}||z|^{\frac{\rho\alpha}{s}}|dz|\right)^{2}(r'-r)^{2H-2-\epsilon}dr'\int_{r}^{t}(r'-r)^{-1+\epsilon}dr\\
			\leq& C\tau^{2}(t-r)^{\epsilon}\int_{\Gamma^{\tau}_{\theta,\kappa}}|z|^{\frac{2\rho\alpha}{s}+1-2H-\epsilon}|dz|\\
			&\cdot\int_{r}^{t}\int_{\Gamma^{\tau}_{\theta,\kappa}}|e^{2z(t-r')}||z|^{2H-1+\epsilon}|dz|(r'-r)^{2H-2-\epsilon}dr'\\
			\leq& C\tau^{2H-\frac{2\rho\alpha}{s}+\epsilon}(t-r)^{\epsilon}\int_{r}^{t}\int_{\Gamma^{\tau}_{\theta,\kappa}}|e^{2z(t-r')}||z|^{2H-1+\epsilon}|dz|(r'-r)^{2H-2-\epsilon}dr'\\
			\leq&C\tau^{2H-\frac{2\rho\alpha}{s}+\epsilon}(t-r)^{\epsilon}\\
			&\cdot\left (\int_{\Gamma^{\tau}_{\theta,\kappa}}|e^{z(t-r)}||z|^{2\epsilon}|dz|+(t-r)^{2H-1-\epsilon}\int_{\Gamma^{\tau}_{\theta,\kappa}}|e^{Cz(t-r)}||z|^{2H-1+\epsilon}|dz|\right ).
		\end{aligned}
	\end{equation*}
	Due to
	\begin{equation*}
		\begin{aligned}
			&\int_{0}^{t_{n}}\int_{\Gamma^{\tau}_{\theta,\kappa}}|e^{z(t_{n}-r)}||z|^{2\epsilon}|dz|r^{1-2H}dr\\
			\leq&\int_{\Gamma^{\tau}_{\theta,\kappa}}\int_{0}^{\frac{t_{n}}{4}}|e^{z(t_{n}-r)}|r^{1-2H}dr|z|^{2\epsilon}|dz|\\
			&+\int_{\Gamma^{\tau}_{\theta,\kappa}}\int_{\frac{t_{n}}{4}}^{t_{n}}|e^{z(t_{n}-r)}|r^{1-2H}dr|z|^{2\epsilon}|dz|\\
			\leq&C\int_{\Gamma^{\tau}_{\theta,\kappa}}|e^{zt_{n}}||z|^{2H-2+2\epsilon}|dz|+Ct_{n}^{2-2H}\int_{\Gamma^{\tau}_{\theta,\kappa}}|e^{Czt_{n}}||z|^{2\epsilon}|dz|\\
			\leq&C\left(\int_{\Gamma^{\tau}_{\theta,\kappa}}|e^{2zt_{n}}||z|^{4H-3+2\epsilon}|dz|\right )^{\frac{1}{2}}\left(\int_{\Gamma^{\tau}_{\theta,\kappa}}|z|^{-1+2\epsilon}|dz|\right )^{\frac{1}{2}}\\
			&+Ct^{2-2H}_{n}\left(\int_{\Gamma^{\tau}_{\theta,\kappa}}|e^{2Czt_{n}}||z|^{1+2\epsilon}|dz|\right )^{\frac{1}{2}}\left(\int_{\Gamma^{\tau}_{\theta,\kappa}}|z|^{-1+2\epsilon}|dz|\right )^{\frac{1}{2}}\\
			\leq&C\tau^{-\epsilon}t_{n}^{1-2H-\epsilon}
		\end{aligned}
	\end{equation*}
	and
	\begin{equation*}
		\begin{aligned}
			&\int_{0}^{t_{n}}\int_{\Gamma^{\tau}_{\theta,\kappa}}|e^{Cz(t_{n}-r)}||z|^{2H-1+\epsilon}|dz|(t_{n}-r)^{2H-1}r^{1-2H}dr\\
			\leq&\int_{ct_{n}}^{t_{n}}\int_{\Gamma^{\tau}_{\theta,\kappa}}|e^{Cz(t_{n}-r)}||z|^{2H-1+\epsilon}|dz|(t_{n}-r)^{2H-1}r^{1-2H}dr\\
			&+\int_{0}^{ct_{n}}\int_{\Gamma^{\tau}_{\theta,\kappa}}|e^{Cz(t_{n}-r)}||z|^{2H-1+\epsilon}|dz|(t_{n}-r)^{2H-1}r^{1-2H}dr\\
			\leq&\int_{\Gamma^{\tau}_{\theta,\kappa}}\int_{ct_{n}}^{t_{n}}|e^{Cz(t_{n}-r)}|(t_{n}-r)^{2H-1}r^{1-2H}dr|z|^{2H-1+\epsilon}|dz|\\
			&+\int_{\Gamma^{\tau}_{\theta,\kappa}}\int_{0}^{ct_{n}}|e^{Cz(t_{n}-r)}|(t_{n}-r)^{2H-1}r^{1-2H}dr|z|^{2H-1+\epsilon}|dz|\\
			\leq&C\int_{\Gamma^{\tau}_{\theta,\kappa}}|e^{Cz(t_{n}-r)}||z|^{2H-1+\epsilon}|dz|\\
			&+Ct_{n}^{2H-1}\int_{\Gamma^{\tau}_{\theta,\kappa}}\int_{0}^{ct_{n}}|e^{Cz(t_{n}-r)}|r^{1-2H}dr|z|^{2H-1+\epsilon}|dz|\\
			\leq&C\left (\int_{\Gamma^{\tau}_{\theta,\kappa}}e^{-2C|z|t_{n}}|z|^{4H-3}|dz|\right )^{\frac{1}{2}}\left (\int_{\Gamma^{\tau}_{\theta,\kappa}}|z|^{-1+2\epsilon}|dz|\right )^{\frac{1}{2}}\\
			&+Ct_{n}^{2H-1}\int_{\Gamma^{\tau}_{\theta,\kappa}}e^{-C|z|t_{n}}|z|^{4H-3+\epsilon}|dz|\\
			\leq&C\tau^{-\epsilon}t_{n}^{1-2H}+Ct_{n}^{2H-1}\left (\int_{\Gamma^{\tau}_{\theta,\kappa}}e^{-2C|z|t_{n}}|z|^{8H-5}|dz|\right )^{\frac{1}{2}}\left (\int_{\Gamma^{\tau}_{\theta,\kappa}}|z|^{-1+2\epsilon}|dz|\right )^{\frac{1}{2}}\\
			\leq &C\tau^{-\epsilon}t_{n}^{1-2H}
		\end{aligned}
	\end{equation*}
	with $c\in (0,1)$, we have
	\begin{equation*}
		\vartheta_{2}\leq C\tau^{2H-\frac{2\rho\alpha}{s}}.
	\end{equation*}
	Collecting above estimates about $\vartheta_{1}$ and $\vartheta_{2}$ leads to the desired result.
\end{proof}

\section{Spatial discretization and errors analysis}
Now we begin by using the finite element method to discretize the integral fractional Laplacian operator; and then the error estimates for the fully discrete scheme of Eq. \eqref{equretosol} are also provided.

Denote $ X_h $ as piecewise linear finite element space
\begin{equation*}
	X_{h}=\{\nu_h\in C(\bar{\mathbb{D}}): \nu_h|_\mathbf{T}\in \mathcal{P}^1,\  \forall \mathbf{T}\in\mathcal{T}_h,\ \nu_h|_{\partial \mathbb{D}}=0\},
\end{equation*}
where $\mathcal{T}_h$ is a shape regular quasi-uniform partition of the domain $\mathbb{D}$, $h$ is the maximum diameter, and $\mathcal{P}^1$ denotes the set of piecewise polynomials of degree $1$ over $\mathcal{T}_h$. Introduce the $ L^2 $-orthogonal projection $ P_h:\ \mathbb{H}\rightarrow X_h $ \cite{Thomee_2006} by
\begin{equation*}
	\begin{aligned}
		&(P_hu,\nu_h)=(u,\nu_h) \ \quad\forall \nu_h\in X_h;
	\end{aligned}
\end{equation*}
and $\mathcal{A}^{s}_{h}$ is defined by $(\mathcal{A}^{s}_{h}u_{h},\nu_{h})=\langle u_{h}, \nu_{h}\rangle_{s}$ for $u_{h},\nu_{h}\in X_{h}$. The fully discrete Galerkin scheme for Eq. \eqref{equretosol} reads: For every $t\in (0,T]$, find $ u^{n}_{h}\in X_h$ such that
\begin{equation}\label{eqsemischeme}
	\left \{
	\begin{aligned}
		&\left(\frac{u^{n}_h-u^{n-1}_{h}}{\tau},\nu_{h}\right)+\sum_{i=0}^{n-1}d^{(1-\alpha)}_{i}\langle u^{n-i}_h,\nu_{h}\rangle_{s}=\\
		&\qquad\qquad\qquad\qquad\qquad\qquad\qquad\qquad\left (\frac{W^{H}_{Q}(t_{n})-W^{H}_{Q}(t_{n-1})}{\tau},\nu_{h}\right )\quad \forall \nu_{h}\in X_h,
		%(x,t)\in D\times,
		\\
		&u^{0}_h=0. %\qquad x\in D,
		\\
	\end{aligned}
	\right.
\end{equation}
Also, \eqref{eqsemischeme} can be written as
\begin{equation}\label{eqsemischemeop}
	\frac{u^{n}_h-u^{n-1}_{h}}{\tau}+\sum_{i=0}^{n-1}d^{(1-\alpha)}_{i}\mathcal{A}^{s}_{h}u^{n-i}_h=P_{h}\frac{W^{H}_{Q}(t_{n})-W^{H}_{Q}(t_{n-1})}{\tau}.
\end{equation}
Following the derivation of \cite{gunzburger_2018,gunzburger_2019,nie20202}, the solution of \eqref{eqsemischeme} can be written as
\begin{equation}\label{eqrepsolnum}
	\begin{aligned}
		u^{n}_{h}
		=&\int_{0}^{t_{n}}\frac{1}{\tau}\sum_{i=1}^{n}\chi_{(t_{i-1},t_{i}]}(s)\int_{t_{i-1}}^{t_{i}}\bar{\mathcal{R}}_{h}(t_{n}-\xi)P_{h}d\xi dW^{H}_{Q}(s),
	\end{aligned}
\end{equation}
where
\begin{equation*}
	\bar{\mathcal{R}}_{h}(t)=\frac{1}{2\pi\mathbf{i}}\int_{\Gamma^{\tau}_{\theta,\kappa}}e^{zt}(\delta_{\tau}(e^{-z\tau}))^{\alpha-1}((\delta_{\tau}(e^{-z\tau}))^{\alpha}+\mathcal{A}^{s}_{h})^{-1}\frac{z\tau}{e^{z\tau}-1}dz.
\end{equation*}

Similarly, the solution of \eqref{eqfullscheme} can be represented
as
\begin{equation}\label{eqrepsolfull2}
	u^{n}=\int_{0}^{t_{n}}\frac{1}{\tau}\sum_{i=1}^{n}\chi_{(t_{i-1},t_{i}]}(s)\int_{t_{i-1}}^{t_{i}}\bar{\mathcal{R}}(t_{n}-\xi)d\xi dW^{H}_{Q}(s).
\end{equation}

\begin{lemma}\label{lemeroper2}
	Let $\mathcal{A}^{s}$ with  homogeneous Dirichlet boundary condition with $s\in(0,1)$,  and $z\in\Gamma_{\theta,\kappa}$  with $\theta\in(\pi/2,\pi)$. Assume $v\in \hat{H}^{\sigma}(\mathbb{D})$ with $\sigma\in[-s,\frac{1}{2}-s)$. Denote $w=(z^{\alpha}+\mathcal{A}^{s})^{-1}v$ and $w_h=(z^{\alpha}+\mathcal{A}^{s}_{h})^{-1}P_hv$. Then one has
	\begin{equation*}
		\|w-w_h\|_{\mathbb{H}}+h^{\min(s,\frac{1}{2}-\epsilon)}\|w-w_h\|_{\hat{H}^s(\mathbb{D})}\leq Ch^{\gamma}\|v\|_{\hat{H}^{\sigma}(\mathbb{D})},
	\end{equation*}
	where
	\begin{equation*}
		\gamma=\left\{
		\begin{aligned}
			&2s+\sigma,\quad s\in(0,\frac{1}{2}),\\
			&s+\sigma+1/2-\epsilon,\quad s\in[\frac{1}{2},1)
		\end{aligned}\right .
	\end{equation*}
	and
	\begin{equation*}
		(\mathcal{A}^{s}_{h} u_h,v_h)=\langle u_h,v_h\rangle_{s}  \quad\forall u_h,\ v_h\in X_h.
	\end{equation*}
	%    and  $\epsilon>0$ being arbitrarily small.
\end{lemma}

\begin{proof}
	Let $\epsilon>0$ be arbitrarily small. According to the definitions of $w$ and $w_h$, there hold
	\begin{equation*}
		\begin{aligned}
			&z^{\alpha}(w,\chi)+\langle w,\chi\rangle_s=(v,\chi)\quad \forall \chi\in \hat{H}^s(\mathbb{D}),\\
			&z^{\alpha}(w_h,\chi)+\langle w_h,\chi\rangle_s=(v,\chi)\quad \forall \chi\in X_h.\\
		\end{aligned}
	\end{equation*}
	Thus
	\begin{equation*}
		z^{\alpha}(e,\chi)+\langle e,\chi\rangle_s=0\quad \forall \chi\in X_h,
	\end{equation*}
	where $e=w-w_h$.
	Then one has
	\begin{equation*}
		\begin{aligned}
			|z|^{\alpha}\|e\|^2_{\mathbb{H}}+\|e\|^2_{\hat{H}^s(\mathbb{D})}\leq& C\left ||z|^{\alpha}\|e\|^2_{\mathbb{H}}+\|e\|^2_{\hat{H}^s(\mathbb{D})}\right |\\
			=&C\left ||z|^{\alpha}(e,w-\chi)+\langle e,(w-\chi)\rangle_s\right |.
		\end{aligned}
	\end{equation*}
	Taking $\chi=\pi_h w$ as the suitable quasi-interpolation \cite{Acosta20173,Acosta20171} of $w$ and using the Cauchy-Schwarz inequality, we obtain
	\begin{equation*}
		\begin{aligned}
			& |z|^{\alpha}\|e\|^2_{\mathbb{H}}+\|e\|^2_{\hat{H}^s(\mathbb{D})}
			\\	
			& \leq C h^{s+\sigma}|z|^{\alpha}\|e\|_{\mathbb{H}}\|w\|_{\hat{H}^{s+\sigma}(\mathbb{D})}+Ch^{s+\sigma}\|e\|_{\hat{H}^{s}(\mathbb{D})}\|w\|_{\hat{H}^{2s+\sigma}(\mathbb{D})}.
		\end{aligned}
	\end{equation*}
According to Lemma \ref{lemresolest}, there exists 
	\begin{equation*}
		\begin{aligned}
			& |z|^{\alpha}\|e\|^2_{\mathbb{H}}+\|e\|^2_{\hat{H}^s(\mathbb{D})}
			\\
			& \leq Ch^{s+\sigma}\|v\|_{\hat{H}^{\sigma}(\mathbb{D})}\left(|z|^{\frac{\alpha}{2}}\|e\|_{\mathbb{H}}+\|e\|_{\hat{H}^s(\mathbb{D})}\right).
		\end{aligned}
	\end{equation*}
So
	\begin{equation*}
		|z|^{\frac{\alpha}{2}}\|e\|_{\mathbb{H}}+\|e\|_{\hat{H}^{s}(\mathbb{D})}\leq Ch^{s+\sigma}\|v\|_{\hat{H}^{\sigma}(\mathbb{D})}.
	\end{equation*}
	Similarly, for $\phi\in \mathbb{H}$, we set
	\begin{equation*}
		\varphi=(z^{\alpha}+A)^{-1}\phi, \quad \varphi_h=(z^{\alpha}+A_h)^{-1}P_h\phi.
	\end{equation*}
	By a duality argument, one has
	\begin{equation*}
		\|e\|_{\mathbb{H}}= \sup_{\phi\in \mathbb{H}}\frac{|(e,\phi)|}{\|\phi\|_{\mathbb{H}}}=\sup_{\phi\in \mathbb{H}}\frac{|z^{\alpha}(e,\varphi)+\langle e,\varphi\rangle_s|}{\|\phi\|_{\mathbb{H}}}.
	\end{equation*}
	Then, using the fact that $|z|^{\frac{\alpha}{2}}\|\varphi-\varphi_h\|_{\mathbb{H}}+\|\varphi-\varphi_h\|_{\hat{H}^s(\mathbb{D})}\leq Ch^{\min(s,\frac{1}{2}-\epsilon)} \|\phi\|_{\mathbb{H}}$ \cite{Acosta20173}, we have
	\begin{equation*}
		\begin{aligned}
			|z^{\alpha}(e,\varphi) +\langle e,\varphi\rangle_s|=& ||z|^{\alpha}(e,\varphi - \varphi_h) + \langle e,(\varphi - \varphi_h )\rangle_s|\\
			\leq&|z|^{\frac{\alpha}{2}}\|e\|_{\mathbb{H}}|z|^{\frac{\alpha}{2}}\|\varphi -\varphi_h\|_{\mathbb{H}}\\
			&+\|e\|_{\hat{H}^s(\mathbb{D})}\|\varphi - \varphi_h\|_{\hat{H}^s(\mathbb{D})}\\
			\leq& Ch^{\gamma} \|v\|_{\hat{H}^{\sigma}(\mathbb{D})}\|\phi\|_{\mathbb{H}},
		\end{aligned}
	\end{equation*}
	where
	\begin{equation*}
		\gamma=\left\{
		\begin{aligned}
			&2s+\sigma,\quad\qquad\qquad s\in(0,\frac{1}{2}),\\
			&s+\sigma+1/2-\epsilon, \quad s\in[\frac{1}{2},1).
		\end{aligned}\right .
	\end{equation*}
\end{proof}

Besides, introduce fractional Ritz projection $R^{s}_{h}:\hat{H}^{s}(\mathbb{D})\rightarrow X_{h}$ defined by, for $s\in(0,1)$,
\begin{equation*}
	(\mathcal{A}^{s}(u-R^{s}_{h}u),v_{h})=\langle u-R^{s}_{h}u,v_{h}\rangle_{s}=0\quad \forall v_{h}\in X_{h};
\end{equation*}
and it has the following properties.
\begin{lemma}\label{lemritz}
	Let $\mathcal{A}^{s}u=f$ and $\mathcal{A}^{s}_{h}R^{s}_{h}u=P_{h}f$ with $f\in \hat{H}^{\sigma}(\mathbb{D})$ and $\sigma\in [-s,\frac{1}{2}-s)$. Then we have
	\begin{equation*}
		\|R_{h}^{s}u\|_{\hat{H}^{s}(\mathbb{D})}\leq \|u\|_{\hat{H}^{s}(\mathbb{D})}
	\end{equation*}
	and
	\begin{equation*}
		\|R_{h}^{s}u-u\|_{\mathbb{H}}+h^{\min(s,\frac{1}{2}-\epsilon)}\|R_{h}^{s}u-u\|_{\hat{H}^{s}(\mathbb{D})}\leq Ch^{\gamma}\|f\|_{\hat{H}^{\sigma}(\mathbb{D})},
	\end{equation*}
	where
	\begin{equation*}
		\gamma=\left\{
		\begin{aligned}
			&2s+\sigma,\quad\qquad\qquad s\in(0,\frac{1}{2}),\\
			&s+\sigma+1/2-\epsilon,\quad s\in[\frac{1}{2},1).
		\end{aligned}\right .
	\end{equation*}
\end{lemma}
\begin{proof}
	Simple calculations lead to
	\begin{equation*}
		\begin{aligned}
			\|R^{s}_{h}u\|_{\hat{H}^{s}(\mathbb{D})}\leq& C\sqrt{(\mathcal{A}^{s}R^{s}_{h}u,R^{s}_{h}u)}\leq C\sup_{v_{h}\in X_{h}}\frac{(\mathcal{A}^{s}R^{s}_{h}u,v_{h})}{\|v_{h}\|_{\hat{H}^{s}(\mathbb{D})}}\\
			\leq&  C\sup_{v_{h}\in X_{h}}\frac{(\mathcal{A}^{s}u,v_{h})}{\|v_{h}\|_{\hat{H}^{s}(\mathbb{D})}}\leq  C\sup_{v\in \hat{H}^{s}(\mathbb{D})}\frac{(\mathcal{A}^{s}u,v)}{\|v\|_{\hat{H}^{s}(\mathbb{D})}}\leq C\|u\|_{\hat{H}^{s}(\mathbb{D})}.
		\end{aligned}
	\end{equation*}
	Similar to the proof of Lemma \ref{lemeroper2}, we can get the second estimate.
\end{proof}
	Combining the definitions of $\mathcal{A}^{s}$, $\mathcal{A}^{s}_{h}$, and $P_{h}$ results in 
	\begin{equation*}
		\begin{aligned}
			(\mathcal{A}^{s}_{h}R^{s}_{h}u,v_{h})=&\langle R^{s}_{h}u,v_{h}\rangle_{s}=\langle u,v_{h}\rangle_{s}\\
			&=(\mathcal{A}^{s}u,v_{h})=(P_{h}\mathcal{A}^{s}u,v_{h}),
		\end{aligned}
	\end{equation*}
	which leads to
	\begin{equation*}
		\mathcal{A}^{s}_{h}R^{s}_{h}=P_{h}\mathcal{A}^{s}.
	\end{equation*}
	Taking $\mathcal{A}^{s}\varphi=\phi$ and $\mathcal{A}^{s}_{h}R^{s}_{h}\varphi=P_{h}\phi$ with $\phi\in\mathbb{H}$, and $s\in(0,1)$, yields
	\begin{equation*}
		\begin{aligned}
			\|(\mathcal{A}^{s}_{h})^{-\frac{1}{2}}P_{h}\phi\|_{\mathbb{H}}=&((\mathcal{A}^{s}_{h})^{-\frac{1}{2}}P_{h}\phi,(\mathcal{A}^{s}_{h})^{-\frac{1}{2}}P_{h}\phi)^{\frac{1}{2}}\\
			=&(P_{h}\phi,(\mathcal{A}^{s}_{h})^{-1}P_{h}\phi)^{\frac{1}{2}}\\
			=&(\mathcal{A}^{s}_{h}R^{s}_{h}\varphi,R^{s}_{h}\varphi)^{\frac{1}{2}}\\
			=&\|R^{s}_{h}\varphi\|_{\hat{H}^{s}(\mathbb{D})}\\
			\leq&\|\varphi\|_{\hat{H}^{s}(\mathbb{D})}\\
			\leq&C\|A^{-s/2}\phi\|_{\mathbb{H}}.
		\end{aligned}
	\end{equation*}
	By the stability of $L_{2}$ projection and interpolation theory \cite{Adams2013}, we have, for $\phi\in \mathbb{H}$ and $s\in(0,1)$,
	\begin{equation*}
		\|(\mathcal{A}^{s}_{h})^{-\sigma/2}P_{h}\phi\|_{\mathbb{H}}\leq 	\|A^{-s\sigma/2}\phi\|_{\mathbb{H}},\quad \sigma\in [0,1],
	\end{equation*}
leading to
	\begin{equation}\label{eqnagsptest}
		\begin{aligned}
			&\|(z^{\alpha}+\mathcal{A}^{s})^{-1}A^{\sigma}-(z^{\alpha}+\mathcal{A}^{s}_{h})^{-1}P_{h}A^{\sigma}\|\\
			\leq &\|(z^{\alpha}+\mathcal{A}^{s})^{-1}A^{\sigma}\|+\|(z^{\alpha}+\mathcal{A}^{s}_{h})^{-1}P_{h}A^{\sigma}\|\\
			\leq& C|z|^{(\frac{\sigma}{s}-1)\alpha}+\|(z^{\alpha}+\mathcal{A}^{s}_{h})^{-1}(\mathcal{A}^{s}_{h})^{\sigma/s}(\mathcal{A}^{s}_{h})^{-\sigma/s}P_{h}A^{\sigma}\|\\
			\leq& C|z|^{(\frac{\sigma}{s}-1)\alpha},\quad \sigma\in[0,\frac{s}{2}],
		\end{aligned}
	\end{equation}
	where we use Lemma \ref{lemresolest}.
	For $s\in(0,\frac{1}{2})$ and $\phi\in\hat{H}^{2s}(\mathbb{D})$, by Lemma \ref{lemritz}, we obtain
	\begin{equation*}
		\begin{aligned}
			&\|((z^{\alpha}+\mathcal{A}^{s})^{-1}-(z^{\alpha}+\mathcal{A}^{s}_{h})^{-1}R^{s}_{h})\phi\|_{\mathbb{H}}\\
			=&\|z^{-\alpha}(I-\mathcal{A}^{s}(z^{\alpha}+\mathcal{A}^{s})^{-1})\phi-z^{-\alpha}(R^{s}_{h}-\mathcal{A}^{s}_{h}(z^{\alpha}+\mathcal{A}^{s}_{h})^{-1}R^{s}_{h})\phi\|_{\mathbb{H}}\\
			\leq& |z|^{-\alpha}\|(I-R^{s}_{h})\phi\|_{\mathbb{H}}+|z|^{-\alpha}\|((z^{\alpha}+\mathcal{A}^{s})^{-1}-(z^{\alpha}+\mathcal{A}^{s}_{h})^{-1}P_{h})\mathcal{A}^{s}\phi\|_{\mathbb{H}}\\
			\leq& C|z|^{-\alpha}h^{2s}\|\phi\|_{\hat{H}^{2s}(\mathbb{D})},
		\end{aligned}
	\end{equation*}
	which gives
	\begin{equation}\label{eq05spter}
		\begin{aligned}
			&\|((z^{\alpha}+\mathcal{A}^{s})^{-1}-(z^{\alpha}+\mathcal{A}^{s}_{h})^{-1}P_{h})\phi\|_{\mathbb{H}}\\
			\leq&\|((z^{\alpha}+\mathcal{A}^{s})^{-1}-(z^{\alpha}+\mathcal{A}^{s}_{h})^{-1}R^{s}_{h})\phi\|_{\mathbb{H}}\\
			&+\|((z^{\alpha}+\mathcal{A}^{s}_{h})^{-1}R^{s}_{h}-(z^{\alpha}+\mathcal{A}^{s}_{h})^{-1}P_{h})\phi\|_{\mathbb{H}}\\
			\leq& Ch^{2s}|z|^{-\alpha}\|\phi\|_{\hat{H}^{2s}(\mathbb{D})}.
		\end{aligned}
	\end{equation}

\begin{theorem}\label{thmspaceerror}
	Let $u^{n}$ and $u^{n}_{h}$ be the solutions of \eqref{eqfullscheme} and \eqref{eqsemischeme}, respectively. For $\rho\in[\max(\frac{s}{2}-\frac{1}{4}+\epsilon,-s),\min(\frac{s}{2},\frac{sH}{\alpha}-\epsilon)]$, we have
	\begin{equation*}
	\mathbb{E}\|u^{n}-u^{n}_{h}\|_{\mathbb{H}}^{2}\leq
		\left\{
		\begin{aligned}
			&Ch^{\min(\frac{(4\rho-1-2s)(Hs-\alpha\rho)}{\alpha(\rho-s)}-\epsilon,2s-4\rho+1-2\epsilon)},\quad~s\in[\frac{1}{2},1),\\
			&Ch^{\min(\frac{4sH}{\alpha}-4\rho-\epsilon,4s-4\rho)},\quad s\in(0,\frac{1}{2}).
		\end{aligned}\right.
	\end{equation*}
\end{theorem}
\begin{proof}
	Subtracting \eqref{eqrepsolfull2} from \eqref{eqrepsolnum} yields
	\begin{equation*}
		\begin{aligned}
			&\mathbb{E}\|u^{n}-u^{n}_{h}\|^{2}_{\mathbb{H}}\\
			=&\mathbb{E}\left \|\int_{0}^{t_{n}}\frac{1}{\tau}\sum_{i=1}^{n}\chi_{(t_{i-1},t_{i}]}(r)\int_{t_{i-1}}^{t_{i}}\bar{\mathcal{R}}(t_{n}-\xi)-\bar{\mathcal{R}}_{h}(t_{n}-\xi)P_{h}d\xi dW^{H}_{Q}(r)\right \|_{\mathbb{H}}^{2}\\
			=&\mathbb{E}\left \|\int_{0}^{t_{n}}\frac{1}{\tau}\sum_{i=1}^{n}\chi_{(t_{i-1},t_{i}]}(r)\int_{t_{i-1}}^{t_{i}}\mathcal{E}(t_{n}-\xi) d\xi dW^{H}_{Q}(r)\right \|_{\mathbb{H}}^{2},
		\end{aligned}
	\end{equation*}
	where
	\begin{equation*}
		\mathcal{E}(t_{n}-\xi)=\bar{\mathcal{R}}(t_{n}-\xi)-\bar{\mathcal{R}}_{h}(t_{n}-\xi)P_{h}.
	\end{equation*}
	By It\^{o}'s isometry, one has
	\begin{equation*}
		\begin{aligned}
			&\mathbb{E}\|u^{n}-u^{n}_{h}\|^{2}_{\mathbb{H}}\\
			\leq& C\mathbb{E}\Bigg \|\int_{0}^{t_{n}}\int_{r}^{t_{n}}\frac{1}{\tau}\sum_{i=1}^{n}\chi_{(t_{i-1},t_{i}]}(r')\\
			&\qquad\qquad\qquad\qquad\cdot\int_{t_{i-1}}^{t_{i}}\mathcal{E}(t_{n}-\xi)d\xi(r'-r)^{H-\frac{3}{2}}\left(\frac{r}{r'}\right )^{\frac{1}{2}-H}dr'dW_{Q}(r)\Bigg \|_{\mathbb{H}}^{2}\\
			\leq& C\int_{0}^{t_{n}}\left \|\int_{r}^{t_{n}}\frac{1}{\tau}\sum_{i=1}^{n}\chi_{(t_{i-1},t_{i}]}(r')\int_{t_{i-1}}^{t_{i}}\mathcal{E}(t_{n}-\xi)d\xi(r'-r)^{H-\frac{3}{2}}\left(\frac{r}{r'}\right )^{\frac{1}{2}-H}dr'\right \|_{\mathcal{L}_{2}^{0}}^{2}dr\\
			\leq& Ct_{n}^{2H-1}\int_{0}^{t_{n}}\left \|\int_{r}^{t_{n}}\frac{1}{\tau}\sum_{i=1}^{n}\chi_{(t_{i-1},t_{i}]}(r')\int_{t_{i-1}}^{t_{i}}\mathcal{E}(t_{n}-\xi)d\xi(r'-r)^{H-\frac{3}{2}}A^{\rho}dr'\right \|^{2}r^{1-2H}dr.\\
			%\leq& Ct_{n}^{2H-1}\int_{0}^{t_{n}}\left \|\int_{r}^{t_{n}}\frac{1}{\tau}\sum_{i=1}^{n}\chi_{(t_{i-1},t_{i}]}(r')\int_{t_{i-1}}^{t_{i}}\mathcal{E}(t_{n}-\xi)A^{\rho}d\xi(r'-r)^{H-\frac{3}{2}}dr'\right \|^{2}r^{1-2H}dr.
		\end{aligned}
	\end{equation*}
	Simple calculations lead to
	\begin{equation*}
		\begin{aligned}
			&\left \|\frac{1}{\tau}\sum_{i=1}^{n}\chi_{(t_{i-1},t_{i}]}(r')\int_{t_{i-1}}^{t_{i}}\mathcal{E}(t_{n}-\xi)A^{\rho}d\xi\right \|\\
			\leq&C\Bigg \|\frac{1}{\tau}\sum_{i=1}^{n}\chi_{(t_{i-1},t_{i}]}(r')\int_{t_{i-1}}^{t_{i}}\int_{\Gamma_{\theta,\kappa}^{\tau}}e^{z(t_{n}-r')}e^{z(r'-\xi)}(\delta_{\tau}(e^{-z\tau}))^{\alpha-1}\frac{z\tau}{e^{z\tau}-1}\\
			&\cdot(((\delta_{\tau}(e^{-z\tau}))^{\alpha}+\mathcal{A}^{s})^{-1}-((\delta_{\tau}(e^{-z\tau}))^{\alpha}+\mathcal{A}^{s}_{h})^{-1}P_{h})A^{\rho}dzd\xi \Bigg \|\\
			\leq&C\Bigg \|\int_{\Gamma_{\theta,\kappa}^{\tau}}e^{z(t_{n}-r')}(\delta_{\tau}(e^{-z\tau}))^{\alpha-1}\frac{z\tau}{e^{z\tau}-1}\\
			&\qquad\qquad\cdot((\delta_{\tau}(e^{-z\tau}))^{\alpha}+\mathcal{A}^{s})^{-1}-((\delta_{\tau}(e^{-z\tau}))^{\alpha}+\mathcal{A}^{s}_{h})^{-1}P_{h})A^{\rho} dz\Bigg \|.\\
		\end{aligned}
	\end{equation*}
	For $\rho\in[\max(\frac{s}{2}-\frac{1}{4}+\epsilon,0),\frac{s}{2})$, using \eqref{eqnagsptest}, Lemma \ref{lemeroper2}, and interpolation properties, we find
	\begin{equation*}
		\begin{aligned}
		&	\left \|\frac{1}{\tau}\sum_{i=1}^{n}\chi_{(t_{i-1},t_{i}]}(r')\int_{t_{i-1}}^{t_{i}}\mathcal{E}(t_{n}-\xi)A^{\rho}d\xi\right \|\\
			&\qquad\leq Ch^{(1-\beta)\gamma}\int_{\Gamma_{\theta, \kappa}^{\tau}}|e^{z(t_{n}-r')}||z|^{\alpha-1+\beta(\frac{\rho}{s}-1)\alpha}|dz|
		\end{aligned}
	\end{equation*}
	with $\beta\in[0,1]$
	and
	\begin{equation*}
		\gamma=\left\{
		\begin{aligned}
			&2s-2\rho,\quad s\in (0,\frac{1}{2}),\\
			&s-2\rho+1/2-\epsilon,\quad s\in [\frac{1}{2},1).
		\end{aligned}\right .
	\end{equation*}
	Thus
	\begin{equation*}
		\begin{aligned}
			&\mathbb{E}\|u^{n}-u^{n}_{h}\|^{2}_{\mathbb{H}}\\
			\leq&Ch^{2(1-\beta)\gamma} t_{n}^{2H-1}\\
			&\cdot\int_{0}^{t_{n}}\left (\int_{r}^{t_{n}}\int_{\Gamma_{\theta, \kappa}^{\tau}}|e^{z(t_{n}-r')}||z|^{\alpha-1+\beta(\frac{\rho}{s}-1)\alpha}|dz|(r'-r)^{H-\frac{3}{2}}dr'\right )^{2}r^{1-2H}dr.\\
		\end{aligned}
	\end{equation*}
Using	mean value theorem leads to
	\begin{equation*}
		\begin{aligned}
			&\int_{r}^{t_{n}}\int_{\Gamma_{\theta, \kappa}^{\tau}}|e^{z(t_{n}-r')}||z|^{\alpha-1+\beta(\frac{\rho}{s}-1)\alpha}|dz|(r'-r)^{H-\frac{3}{2}}dr'\\
			\leq&\int_{0}^{t_{n}-r}\int_{\Gamma_{\theta, \kappa}^{\tau}}|e^{z(t_{n}-r-\xi)}||z|^{\alpha-1+\beta(\frac{\rho}{s}-1)\alpha}|dz|\xi^{H-\frac{3}{2}}d\xi\\
			\leq&\int_{\Gamma_{\theta, \kappa}^{\tau}}|e^{z(t_{n}-r)}|\int_{0}^{\frac{t_{n}-r}{4}}|e^{-z\xi}|\xi^{H-\frac{3}{2}}d\xi|z|^{\alpha-1+\beta(\frac{\rho}{s}-1)\alpha}|dz|\\
			&+\int_{\Gamma_{\theta, \kappa}^{\tau}}|e^{z(t_{n}-r)}|\int_{\frac{t_{n}-r}{4}}^{t_{n}-r}|e^{-z\xi}|\xi^{H-\frac{3}{2}}d\xi|z|^{\alpha-1+\beta(\frac{\rho}{s}-1)\alpha}|dz|\\
			\leq &C(t_{n}-r)^{H-\frac{1}{2}+\beta(1-\frac{\rho}{s})\alpha-\alpha}.
		\end{aligned}
	\end{equation*}
	To preserve the boundness of $\mathbb{E}\|u^{n}-u^{n}_{h}\|^{2}_{\mathbb{H}}$, we need to require $H-\frac{1}{2}+\beta(1-\frac{\rho}{s})\alpha-\alpha>-\frac{1}{2}$, i.e., $\beta\geq\max(\frac{s\alpha-sH}{(s-\rho)\alpha}+\epsilon,0)$. Thus
	\begin{equation*}
		\mathbb{E}\|u^{n}-u^{n}_{h}\|^{2}_{\mathbb{H}}\leq \left\{
		\begin{aligned}
			&Ch^{\min(\frac{(4\rho-1-2s)(Hs-\alpha\rho)}{\alpha(\rho-s)}-\epsilon,2s-4\rho+1-2\epsilon)},\quad~s\in[\frac{1}{2},1),\\
			&Ch^{\min(\frac{4sH}{\alpha}-4\rho-\epsilon,4s-4\rho)},\quad s\in(0,\frac{1}{2}).
		\end{aligned}\right.
	\end{equation*}
	For $\rho<0$, we have $s\in(0,\frac{1}{2})$. Combining Lemma \ref{lemeroper2}, \eqref{eq05spter}, and interpolation properties results in
	\begin{equation*}
		\left \|\frac{1}{\tau}\sum_{i=1}^{n}\chi_{(t_{i-1},t_{i}]}(r')\int_{t_{i-1}}^{t_{i}}\mathcal{E}(t_{n}-\xi)A^{\rho}d\xi\right \|\leq Ch^{2\beta s-2\rho}\int_{\Gamma_{\theta, \kappa}^{\tau}}|e^{z(t_{n}-r')}||z|^{\beta\alpha-1}|dz|
	\end{equation*}
	with $\beta\in[0,1]$. Similarly, there holds
	\begin{equation*}
		\begin{aligned}
			&\int_{r}^{t_{n}}\int_{\Gamma_{\theta, \kappa}^{\tau}}|e^{z(t_{n}-r')}||z|^{\beta\alpha-1}|dz|(r'-r)^{H-\frac{3}{2}}dr'\\
			\leq&\int_{0}^{t_{n}-r}\int_{\Gamma_{\theta, \kappa}^{\tau}}|e^{z(t_{n}-r-\xi)}||z|^{\beta\alpha-1}|dz|\xi^{H-\frac{3}{2}}d\xi\\
			\leq&\int_{\Gamma_{\theta, \kappa}^{\tau}}|e^{z(t_{n}-r)}|\int_{0}^{\frac{t_{n}-r}{4}}|e^{-z\xi}|\xi^{H-\frac{3}{2}}d\xi|z|^{\beta\alpha-1}|dz|\\
			&+\int_{\Gamma_{\theta, \kappa}^{\tau}}|e^{z(t_{n}-r)}|\int_{\frac{t_{n}-r}{4}}^{t_{n}-r}|e^{-z\xi}|\xi^{H-\frac{3}{2}}d\xi|z|^{\beta\alpha-1}|dz|\\
			\leq &C(t_{n}-r)^{H-\frac{1}{2}-\beta\alpha}.
		\end{aligned}
	\end{equation*}
	To preserve the boundness of $\mathbb{E}\|u^{n}-u^{n}_{h}\|^{2}_{\mathbb{H}}$, we need to require $H-\frac{1}{2}-\beta\alpha>-\frac{1}{2}$, i.e., $\beta\leq\min(\frac{H}{\alpha}-\epsilon,1)$. So we get
	\begin{equation*}
		\mathbb{E}\|u^{n}-u^{n}_{h}\|^{2}_{\mathbb{H}}\leq Ch^{\min(\frac{4sH}{\alpha}-4\rho-\epsilon,4s-4\rho)}.
	\end{equation*}
\end{proof}
\section{Numerical experiments}
In this part, we present some examples to verify the theoretical results in Theorems \ref{thmtimeerror} and \ref{thmspaceerror} with different $s$, $\alpha$, $H$, and $\rho$.
Suppose that the covariance operator $Q$ shares the eigenfunctions with the operator $A$ and denote its eigenvalues as $\varLambda_{k}=k^{m}$, where $k=1,2,\cdots$, and $m\leq 0$.
By the assumption $\|A^{-\rho}\|_{\mathcal{L}^{0}_{2}}<\infty$ and Lemma \ref{thmeigenvalue}, we have $\rho>\frac{1+m}{4}d$ and $d$ is the dimension of the space.

For convenience, we choose the domain $\mathbb{D}=(0,1)$ and take
\begin{equation*}
	W^{H}_{Q}(x,t)=\sum_{k=1}^{5000}\sqrt{\varLambda_{k}}\phi_k(x)W^{H}_k(t),
\end{equation*}
where
\begin{equation*}
	\phi_{k}(x)=\sqrt{2}\sin(k\pi x).
\end{equation*}
We use $100$ trajectories to compute the solution of Eq. $\eqref{equretosol}$. Due to that the exact solution $u$ is unknown, we calculate
\begin{equation*}
	\begin{aligned}
		&e_{h}=\left (\frac{1}{100}\sum_{i=1}^{100}\|u^{N}_{h}(\omega_{i})-u^{N}_{h/2}(\omega_{i})\|^{2}_{\mathbb{H}}\right )^{\frac{1}{2}},\\
		&e_{\tau}=\left (\frac{1}{100}\sum_{i=1}^{100}\|u_{\tau}(\omega_{i})-u_{\tau/2}(\omega_{i})\|^{2}_{\mathbb{H}}\right )^{\frac{1}{2}}
	\end{aligned}
\end{equation*}
to measure the spatial and temporal errors, where the $u^{N}_{h}(\omega_{i})$ ($u_{\tau}(\omega_{i})$) means the numerical solution of $u$ at $t_N$ with mesh size $h$ (step size $\tau$) and trajectory $\omega_{j}$; so the spatial and temporal convergence rates can be, respectively, tested by
\begin{equation*}
	{\rm Rate}=\frac{\ln(e_{h}/e_{h/2})}{\ln(2)},\quad {\rm Rate}=\frac{\ln(e_{\tau}/e_{\tau/2})}{\ln(2)}.
\end{equation*}

\begin{example}
	In this example, we solve Eq. \eqref{equretosol} numerically with terminal time $T=1$ by numerical scheme \eqref{eqsemischeme} to validate the temporal convergence rates. Here, we take $h=\frac{1}{256}$ to make the error incurred by spatial discretization negligible. The numerical results with different $m$, $\alpha$, $s$, $H$ are presented in Table \ref{tab:time}. All the results agree with the predicted theoretical convergence rates  (the
	numbers in the bracket in the last column) by Theorem \ref{thmtimeerror}.

\begin{table}[htbp]
	\caption{Temporal errors and convergence rates}
	\begin{tabular}{ccccccc}
		\hline
		m & $(\alpha,s,H)\backslash \tau$ & 1/32 & 1/64&1/128 & 1/256 & Rates \\
		\hline
		0 & (0.7,0.6,0.85) & 8.914E-03 & 6.174E-03&4.157E-03 & 2.833E-03 & 0.5513(0.5583) \\ &(0.8,0.7,0.85) & 9.150E-03 & 5.971E-03&4.084E-03 & 2.755E-03 & 0.5773(0.5643) \\
		-0.5 & (0.4,0.3,0.8) & 7.549E-03 & 4.754E-03&2.976E-03 & 1.851E-03 & 0.6761(0.6333) \\
		& (0.6,0.8,0.6) & 1.356E-02 & 9.439E-03&6.715E-03 & 4.722E-03 & 0.5075(0.5063) \\
		-1 & (0.3,0.4,0.8) & 2.766E-03 & 1.640E-03&9.619E-04 & 5.559E-04 & 0.7716(0.8000) \\
		& (0.8,0.6,0.6) & 1.878E-02 & 1.252E-02&8.733E-03 & 5.754E-03 & 0.5690(0.6000) \\
		\hline
	\end{tabular}
	\label{tab:time}
\end{table}
\end{example}
\begin{example}
	Here, we perform some numerical experiments to validate the spatial convergence rates.  We take $T=0.01$ and $\tau=\frac{T}{1024}$ to eliminate the influence from temporal discretization. We choose different $m$, $\alpha$, $s$, $H$; and the corresponding numerical results with different $s\in(0,\frac{1}{2})$ and $s\in[\frac{1}{2},1)$  are presented in Tables \ref{tab:spa1} and \ref{tab:spa2}, respectively, which verify the results of Theorem \ref{thmspaceerror}.
\begin{table}[htbp]
	\caption{Spatial errors and convergence rates with $s\in(0,1/2)$}
	\begin{tabular}{ccccccc}
		\hline
		m & $(\alpha,s,H)\backslash h$ & 1/64 & 1/128 & 1/256 & 1/512 &Rates\\
		\hline
		-0.5 & (0.3,0.3,0.7) & 4.250E-02 & 3.461E-02 & 2.818E-02 & 2.248E-02&0.3062(0.3500) \\
		& (0.3,0.4,0.8) & 9.721E-03 & 6.738E-03 & 4.619E-03 & 3.162E-03&0.5401(0.5500) \\
		
		-1 & (0.5,0.3,0.7) & 1.965E-02 & 1.401E-02 & 9.944E-03 & 6.967E-03&0.4988(0.6000) \\
		& (0.5,0.4,0.8) & 1.280E-02 & 7.714E-03 & 4.541E-03 & 2.654E-03&0.7564(0.8000) \\
		-1.5 & (0.9,0.2,0.6) & 9.883E-02 & 7.938E-02 & 6.434E-02 & 5.297E-02&0.2999(0.2667) \\
		\hline
	\end{tabular}
	\label{tab:spa1}
\end{table}

\begin{table}[htbp]
	\caption{Spatial errors and convergence rates with $s\in[1/2,1)$}
	\begin{tabular}{ccccccc}
		\hline
		m & $(\alpha,s,H)\backslash h$ & 1/32 & 1/64 & 1/128 & 1/256 &  Rates\\
		\hline
		-0.2 & (0.5,0.6,0.6) & 4.429E-02 & 2.612E-02 & 1.592E-02 & 9.189E-03 & 0.7563(0.7000) \\
		& (0.5,0.7,0.8) & 6.380E-03 & 3.279E-03 & 1.580E-03 & 7.913E-04 & 1.0038(0.8000) \\
		-0.4 & (0.7,0.6,0.6) & 8.333E-02 & 5.437E-02 & 3.280E-02 & 1.992E-02 & 0.6882(0.6476) \\
		& (0.7,0.7,0.8) & 1.158E-02 & 5.669E-03 & 2.770E-03 & 1.301E-03 & 1.0516(0.9000) \\
		-0.6 & (0.9,0.6,0.6) & 1.303E-01 & 9.042E-02 & 5.948E-02 & 3.896E-02 & 0.5808(0.5400) \\
		\hline
	\end{tabular}
	\label{tab:spa2}
\end{table}
\end{example}
\section{Conclusions}

The macroscopic descriptions for the competition between subdiffusion and L\'evy flights are governed by the fractional Fokker-Planck equation with temporal and spatial fractional derivatives. We do the numerical analyses for the stochastic version of the model, which are driven by the external fractional Gaussian noise. The backward Euler convolution quadrature and finite element method are, respectively, used to approximate the time and spatial operators. The complete error analyses are provided; and numerical experiments verify the effectiveness of the presented numerical scheme.

%In order to describe the effect of random disturbance on physical system, the corresponding stochastic partial differential equations are derived. Also,  how to solve them effectively has been a hot research problem. In this work, we solve the stochastic time-space fractional diffusion equation driven by fractional Gaussian noise numerically, based on the backward Euler convolution quadrature and finite element methods. The complete error analyses are also provided. Numerical experiments verify the effectiveness of our numerical scheme.
%%    Bibliographies can be prepared with BibTeX using amsplain,
%    amsalpha, or (for "historical" overviews) natbib style.
\bibliographystyle{amsplain}
%    Insert the bibliography data here.
\bibliography{references}
\end{document}